\documentclass[reqno,11pt]{amsart}
\usepackage{latexsym}
\usepackage{amsfonts}
\usepackage{amssymb}
\usepackage{mathrsfs}
\usepackage[all]{xy}
\usepackage{bbm}
\usepackage{xcolor}
\usepackage{ulem}
\usepackage[mathscr]{euscript}

\usepackage{enumerate}

\newtheorem{theorem}{Theorem}[section]
\newtheorem{corollary}[theorem]{Corollary}
\newtheorem{proposition}[theorem]{Proposition}
\newtheorem{lemma}[theorem]{Lemma}
\newtheorem{remark}[theorem]{Remark}

\theoremstyle{definition}
\newtheorem{definition}{Definition}
\newtheorem{assumption}{Assumption}

\newcommand{\Q}{{\mathbb {Q}}}
\newcommand{\R}{{\mathbb{R}}}
\newcommand{\Z}{{\mathbb{Z}}}
\newcommand{\N}{{\mathbb{N}}}
\newcommand{\vre}{{\varepsilon}}
\newcommand{\ssm}{{\smallsetminus}}
\newcommand {\ignore}[1] {}

\newcommand{\const}{\operatorname{const}}

\newcommand{\Leb}{\operatorname{Leb}}
\newcommand{\Id}{\operatorname{Id}}
\newcommand{\diameter}{\operatorname{diam}}

\numberwithin{equation}{section}

\newcommand {\comm}[1]   {\textcolor{red}{#1}}
\newcommand {\mynew}[1]   {\textcolor{brown}{#1}}
\newcommand {\mycomm}[1]   {\textcolor{purple}{#1}}
\newcommand\eq[2]{
\begin{equation}
\label{eq:#1}
{#2}
\end{equation}
}
\newcommand{\equ}[1]{\eqref{eq:#1}}
\newcommand{\supp}{\operatorname{supp}}
\newcommand{\dist}{\operatorname{dist}}

\title[Recurrence under Lipschitz Shifts]{Dynamical Borel--Cantelli Lemma for Recurrence under Lipschitz Twists}

\author{Dmitry Kleinbock}
\address{Department of Mathematics, Brandeis University, Waltham MA}
\email{kleinboc@brandeis.edu}

\author{Jiajie Zheng}
\address{Department of Mathematics, Brandeis University, Waltham MA}
 \email{zhengjiajie@brandeis.edu}

\begin{document}

	\begin{abstract}
In the study of some dynamical systems  the limsup set of a sequence of measurable sets is often of interest. The shrinking targets and recurrence are two of the most commonly studied problems that concern limsup sets. However, the zero-one laws for the shrinking targets and recurrence are usually treated separately and proved differently. In this paper, we introduce a generalized definition that can specialize into the shrinking targets and recurrence; our approach gives a unified proof of the zero-one laws for the two problems. 
\end{abstract}

%\subjclass{11J04; 11J13, 37A17, 37D40}
\thanks{D.K.\ was  supported by  NSF grant  DMS-1900560. This material is based upon work supported by a grant from the Institute for Advanced Study School of Mathematics.
}
\date{October 19, 2022}

%\subjclass[2010]{11J13; 11J83, 11H06, 37A17}

\maketitle

%\noindent {New changes}\\
%\comm{Comments}

\section{Introduction}

Throughout the paper, let $(X,d)$ be a separable and compact metric space, and let $(X,\mu,T)$ be a probability measure preserving system. % where $\mathcal{B}$ is {the} Borel $\sigma$-algebra of $X$ and $\mu(X)=1$.
One of the most fundamental results in ergodic theory is the Poincar\'{e} Recurrence Theorem, see e.g.\  \cite[Theorem 2.11]{EW}, which asserts that almost all points in measurable dynamical systems return close to themselves under a measure-preserving map; namely, that 
\begin{equation}\label{prt}
\mu(R_T) = 1,
\end{equation}
where $R_T$ is the \textsl{set of recurrence} for $T$:
$$R_T:=\{x\in X:\liminf_{n\to\infty}d(T^nx,x)=0\}.$$
%\begin{theorem}\label{prt}
%has full measure.
% $\mu$-almost every $x\in X$. 
%One is often interested in a quantitative version of the above result. Namely, 

One of the first results concerning the speed of recurrence is due to Boshernitzan in \cite{bosh}. 
%\begin{theorem}
	Namely, assume that the $\alpha$-dimensional Hausdorff measure of $X$ is  {zero} for some $\alpha>0$. %\comm{If $\mathcal{H}^\alpha$ is $\sigma$-finite, then $\liminf_{n\to\infty}n^{1/\alpha}d(T^nx,x)<\infty$.} 
	Then
	 $$\liminf_{n\to\infty}n^{1/\alpha}d(T^nx,x)=0$$
%\begin{theorem}\label{prt}
for $\mu$-almost every $x\in X$. In other words, for a function
 $\psi:\mathbb{N}\to \mathbb{R}^+$  let us define the following set: 
\begin{equation*}
R_T(\psi):=\left\{x\in X:d(T^n x,x)<\psi(n) \text{ for infinitely many } n\in \mathbb{N}\right\}.
\end{equation*}
Then the Poincar\'{e} Recurrence Theorem says that the set  $$R_T = \bigcap_{\vre > 0} R_T(\vre 1_\N)$$ has full measure, and, with the notation \eq{psi1}{\psi_s(x) := x^{-s},}Bosherniztan's result   says that   $R_T(\vre\psi_{1/\alpha})$ has full measure for any $\vre > 0$ and for any $\alpha$ such that $\mathcal{H}^\alpha(X) = 0$.
%\end{theorem} 

It is a natural problem to find necessary and sufficient conditions on $\psi$ to guarantee that the set $R_T(\psi)$ has measure zero or one. 
{In fact, under some additional assumptions %on the rate of mixing of $T$ 
one expects this condition to be the convergence/divergence of the sum of measures of the sets 
\begin{equation}\label{recsets}
	A_T(n,\psi):=\left\{x\in X:d(T^n x,x)<\psi(n)\right\}.
\end{equation}}
And indeed this was  proved in several special cases {such as \cite{BF,CWW, hussain}; see also \cite{KKP, DFL, Pe} for similar results.}  %\comm{and other references}.

Note that a topic closely related to recurrence is the so-called \textsl{shrinking target problem}, which is concerned with determining the speed at which the orbit
of a $\mu$-typical point accumulates near a fixed point $y\in X$. More precisely, for  $y\in X$ one can define the set
$$R_T^y:=\left\{x\in X:\liminf_{n\to\infty}d(T^nx,y)=0\right\},$$
and, more generally, for a function
 $\psi:\mathbb{N}\to \mathbb{R}^+$ define
\begin{equation*}
R_T^y(\psi):=\left\{x\in X:d(T^n x,y)<\psi(n) \text{ for infinitely many } n\in \mathbb{N}\right\}.
\end{equation*}
Equivalently, letting  $B(x,r)$ stand for the open ball in $X$ centered in $x$ of radius $r$, we can write $R_T^y(\psi) =   \limsup A^y(n,\psi)$, where
\begin{equation}\label{shrsets}
	A_T^y(n,\psi):=\left\{x\in X:d(T^n x,y)<\psi(n)\right\} = T^{-n}B\big(y,\psi(n)\big).
\end{equation}
%and $B(y,r)$ stands for the open ball centered at $y\in X$ and radius $r>0$.
%If we assume that $X$ is separable and $\supp\mu$ is dense in $X$, then  c
Clearly 
\begin{equation}\label{erg}
\mu(R_T^y) = 1\text{ for any $y\in\supp\mu$  \ \ if $T$ is ergodic};
\end{equation}
 furthermore, there have been plenty of results in the literature giving  $0$--$1$ laws for  $\mu\big(R_T^y(\psi)\big)$. In fact, one can often use mixing properties of $T$ to conclude that $\mu\big(R_T^y(\psi)\big)$ is equal to zero/one if and only if the series  
$$\sum_{n=1}^\infty\mu\big(A_T^y(n,\psi)\big) = \sum_{n=1}^\infty\mu\Big(B\big(y,\psi(n)\big)\Big)$$
converges/diverges. See \cite{Ph, CK, KM, FMP,  HNPV}  and many other references. 

The goal of the current paper is to study a property unifying these two settings, and to prove a zero--one law applying to both. Namely, for a Borel measurable function $f:X\to X$ define $R_T^f$, the set of \textsl{$f$-twisted recurrent} points for $T$, by  $$R_T^f:= \left\{x\in X: \liminf_{n\to\infty}d\big(T^n x,f(x)\big)=0\right\}.$$ 
The two previous settings correspond to $f$ being the identity and constant functions respectively.
We will show in the next section that $\mu(R_T^f) = 1$ for any measurable $f$ if $T$ is ergodic and $\mu$ has full support. Furthermore, one can study the rate of twisted recurrence as follows: 
for 
 $\psi:\mathbb{N}\to \mathbb{R}^+$ define
\begin{equation}\label{twisted}
R_T^f(\psi):=\left\{x\in X\left|\begin{aligned}d\big(T^n x,f(x) \big)<\psi(n) \ \ \\ \text{ for infinitely many } n\in \mathbb{N}\end{aligned}\right.\right\},
\end{equation}
so that $R_T^f =  \bigcap_{\vre > 0} R_T^f(\vre 1_X)$. 
In general  the rate of 
twisted recurrence can be arbitrary slow, see \S\ref{more} for examples.   The main goal of the paper is to prove, under assumptions similar to those of \cite{hussain}, a zero--one law for the sets $R_T^f(\psi)$ for a large class of functions $f$.

\smallskip
To state the main result of the paper, we need to adapt  and modify  the  settings and assumptions  from \cite{hussain}. Throughout the paper we write $a\lesssim b$ if $a\leq Cb$ for some constant $C>0$, and $a\asymp b$ if $a\lesssim b$ and $b\lesssim a$.

\smallskip
%Namely, let us suppose 
Our main assumption is that  there exist at most countably many pairwise disjoint open subsets $X_i$, $i\in\mathcal{I}$, of $X$ such that $T|_{X_i}$ is continuous and  injective for each $i$, and $\mu(X\ssm \cup_iX_i) = 0$. Those will be called \textsl{cylinders of order $1$}.
	Then for any $m\in\N$ one can define 
	\begin{equation}\label{cylinders}
	\mathcal{F}_m:=\left\{X_{i_1}\cap T^{-1}X_{i_2}\cap \cdots \cap T^{-(m-1)}X_{i_{m}}: i_1,\dots,i_{m}\in \mathcal{I}\right\}
\end{equation}
 to be the collection of \textsl{cylinders of order $m$}.
Note that for $J\in \mathcal{F}_m$ and $x,y\in J$, the points $T^n x$ and $T^n y$ are in the same partition set $X_i$ for $0\leq n<m$, and hence $T,\ldots,T^m$ are injective on $J$. Also, since $T$ is continuous, each cylinder in $\mathcal{F}_m$ is open. 

% and will impose the following additional assumptions:

\smallskip

%Our first assumption will be that 
Now let us list our assumptions on the measure $\mu$. The first one  is \underbar{Ahlfors regularity} of dimension $\delta>0$; namely, that there exist positive real numbers $\eta_1,\eta_2,r_0$ such that 
	\begin{equation}\label{ar}
		\eta_1r^\delta \leq \mu(B(x,r))\leq \eta_2 r^\delta  \text{ for any ball $B(x,r)\subset X$ with }  {0<r<r_0}.	\end{equation}
		As a consequence, since $\mu$ was assumed to be a probability measure, the space $X$ has finite diameter.  
		
Next, we assume that $(X,\mu,T)$ is  \underbar{uniformly mixing} (a property introduced in \cite{FMP}), that is: there exist  
%constants $K_1>0$ an $0<\gamma<1$ 
a summable
sequence of positive real numbers $(a_n)_{n\in \N}$ such
such that 
	\begin{equation}\label{em}
	\begin{aligned}
		\left|\mu (E\cap T^{-n}F)-\mu(E)\mu(F)\right|\leq a_n \mu(F)\ \ \\ \text{for any balls $E,F\subset X $ and for all }n\geq 1.
		\end{aligned}
	\end{equation}

Note that it was proved in \cite{FMP} that under the aforementioned mixing assumption, for any $y\in X$ and any $\psi$ the set $R_T^y(\psi)$ is null (resp., conull) if the series
$$ \sum_{n=1}^\infty\mu\big(B\big(y,\psi(n)\big)\big) \underset{\eqref{ar}}\asymp \sum_{n=1}^\infty\psi(n)^\delta$$
converges (resp., diverges). However, in order to similarly treat the sets $R_T^f(\psi)$ for more general functions $f$ we will require some more %assumptions 
information on the expanding properties of $T$. For a $m$-cylinder $J$, we define $$K_{J}:=\inf_{x,y\in J,\,x\neq y}\frac{d(T^mx,T^my)}{d(x,y)},$$ and  impose the following additional assumptions:

\ignore{\smallskip
Namely, let us suppose that  there exist at most countably many pairwise disjoint open subsets $X_i$, $i\in\mathcal{I}$, of $X$ such that $T|_{X_i}$ is continuous and  injective for each $i$, and $\mu(X\ssm \cup_iX_i) = 0$. Those will be called \textsl{cylinders of order $1$}.
	Then define $\mathcal{F}_m$ to be the collection of \textsl{cylinders of order $m$},
\begin{equation*}
	\mathcal{F}_m:=\left\{X_{i_0}\cap T^{-1}X_{i_1}\cap \cdots \cap T^{-(m-1)}X_{i_{m-1}}: i_0,i_1,\ldots,i_{m-1}\in \mathcal{I}\right\}.
\end{equation*}
Note that for $J\in \mathcal{F}_m$ and $x,y\in J$, the points $T^n x$ and $T^n y$ are in the same partition set $X_i$ for $0\leq n<m$, and hence $T,\ldots,T^m$ are injective on $J$. Also, since $T$ is continuous, each cylinder in $\mathcal{F}_m$ is open. }

%As a quantitative way to measure the expansiveness of  $T^m$ on  $J\in \mathcal{F}_m$, we will 
\begin{itemize}
\item \underbar{Bounded distortion}: There exists a constant  $K_1>0$ such that 
%for all $x,y,z\in J_m\in \mathcal{F}_m$, with $x\neq y$ and $x\neq z$, 
	\begin{equation}\label{eq10}
	\begin{aligned}
		&K_1^{-1}\leq \frac{d(T^mx,T^my)/d(x,y)}{d(T^mx,T^mz)/d(x,z)}\leq K_1\\
		\text{for all }m\in\N&\text{ and }  x,y,z\in J \in \mathcal{F}_m\text{ with $x\neq y$ and }x\neq z.
		\end{aligned}
\end{equation}

\item \underbar{Expanding properties}:  %There exists a universal constant $L>0$ so that 
	\begin{equation}\label{toinfty}\inf_{J \in \mathcal{F}_m}K_{J }\to \infty\text{ as }m\to\infty 
\end{equation}
and
\begin{equation}\label{qi}\sup_{m\in\N}\sum_{J \in \mathcal{F}_m}K_{J }^{-\delta}< \infty. 
\end{equation} % \comm{I think (1.11) is redundant. See Lemma 4.2.}

\item \underbar{Conformality}: 
	There exists a constant $K_2\geq 1$ such that %for any ball $B(x,r)\subset J_m \in \mathcal{F}_m$, 
	\begin{equation}\label{conf}
	\begin{aligned}
	B(T^mx,K_2^{-1}K_{J}r)\subset T^mB(x,r)\subset B(T^mx,K_2K_{J}r)	\\\text{ for any } m\in\N\text{ and any ball }B(x,r)\subset J \in \mathcal{F}_m.\ \ 
	\end{aligned}
\end{equation}
\end{itemize}

\begin{remark}\label{constants} \rm Notice that the bounded distortion condition \eqref{eq10} implies the second inclusion in  \eqref{conf} with $K_2$ replaced by $K_1$. 
However the first inclusion there does not automatically follow from  \eqref{eq10}, hence the need for an additional condition.  \end{remark}

\begin{remark}\label{comparison} \rm We note that conditions \eqref{ar}--\eqref{conf} are essentially equivalent to Conditions I--V  from \cite{hussain}. Namely:
\begin{itemize}
\item
 \eqref{ar} is a slightly weaker version of  \cite[Condition I]{hussain}.
   \item
 \eqref{em} replaces  \cite[Condition II]{hussain} where the rate of mixing was assumed to be exponential.
   \item As for 
 \eqref{eq10}--\eqref{conf}, in \cite{hussain} the standing assumption was that the restriction of $T$ to $X_i$ for every $i$ is differentiable and expanding, namely  it was assumed that \begin{equation}\label{exp}{\|D_x(T^{-1})\|^{-1} > 1\text{ for any }x\in\cup_iX_i.}\end{equation} 
The role of  \eqref{eq10} was played there by \cite[Condition III]{hussain} stated as follows: there exists a constant  $K_1>0$ such that 
%for all $x,y,z\in J_m\in \mathcal{F}_m$, with $x\neq y$ and $x\neq z$, 
	\begin{equation*}\label{eq10hussain}
	%\begin{aligned}
		K_1^{-1}\leq \frac{d(T^mx,T^my)}{d(x,y)\|D_xT^m\|}\leq K_1 \ 
		\forall \, m\in\N \text{ and } \forall \, x,y\in J \in \mathcal{F}_m\text{ with }x\neq y.
		%\end{aligned}
\end{equation*}
  \item Similarly, the constant $K_J$ for $J \in \mathcal{F}_m$ was defined in  \cite{hussain} by $K_{J}:=\inf_{x \in J}\|D_xT^m\|$, and the role of \eqref{toinfty} was played   by $\inf_{J\in \mathcal{F}_m} K_{J} > 1$ for some
$m\in\N$, which, in view of \eqref{exp}, is easily seen to be equivalent to  $\inf_{J \in \mathcal{F}_m}K_{J }\to \infty\text{ as }m\to\infty $.
Conditions IV and V of \cite{hussain} are identical to \eqref{qi} and \eqref{conf} respectively.
 \end{itemize}
 \end{remark}

Examples of dynamical systems satisfying conditions \eqref{ar}--\eqref{conf} include, as mentioned in  \cite{hussain},   $\beta$-transformations $$M_\beta:x\mapsto \beta x\mod{1}$$ of the unit interval, where $\beta \in\R_{> 1}$, as well as the Gauss map. In \S\ref{example} we add another example to the list:  expanding maps defined by  systems of contracting similarities with the open set condition.
 
Let us now specify the class of functions $f$ which we can treat by our technique. Say that $f:X\to X$ is \textsl{Lipschitz} if $$\sup_{x,y\in X,\,x\neq y}\frac{d\big(f(x),f(y)\big)}{d(x,y)} <\infty,$$ and that $f$ is \textsl{piecewise Lipschitz} if there exist at most countably many %pairwise disjoint 
measurable subsets $Y_i$ of $X$ and Lipschitz functions $f_i:X\to X$, $i\in\mathcal{I}$, such that $\mu(X\ssm \cup_iY_i) = 0$ and $f|_{Y_i} = f_i$ for each $i$.
An example: when $X = [0,1]$, the function $f(x) = \sqrt{x}$ is piecewise Lipschitz but not  Lipschitz.

\smallskip
Now we are ready to state our main theorems.

\begin{theorem}\label{thm2.4}
	 Assume that $(X,\mu,T)$ satisfies conditions \eqref{ar}--\eqref{conf}. Then for any function $\psi:\mathbb{N}\to \mathbb{R}_+$ with $\lim_{n\to \infty}\psi(n)=0$ and any piecewise Lipschitz function $f:X\to X$,  the set $R_T^f(\psi)$ is null %(resp., conull)
	  if and only if the series
\eq{series}{ \sum_{n=1}^\infty\psi(n)^\delta}
converges.
% (resp., diverges). 
\end{theorem}

%\comm{Examples here?}

It is natural to expect that  Theorem \ref{thm2.4} can be strengthened to the full measure of $R_T^f(\psi)$ in the case when %$\sum_{n=1}^\infty \psi(n)^\delta=\infty$ 
the series
\equ{series} diverges. This was done in \cite{hussain} in the case $f=\Id_X$.
Unfortunately for an arbitrary Lipschitz function $f$ %and an arbitrary system satisfying  \eqref{ar}--\eqref{conf} 
the full measure conclusion is outside of our reach.
In the following theorems we handle several special cases. First, employing an argument from \cite{hussain}, we prove

\begin{theorem}\label{Thm2}
	Let $(X,\mu,T)$, $\psi:\mathbb{N}\to \mathbb{R}$ and $f:X\to X$ be as in Theorem \ref{thm2.4}. Furthermore, assume that \eq{comm}{T\circ f=f\circ T.} Then $\mu\big(R_T^f(\psi)\big)=1$
	%$\sum_{n=1}^\infty \psi(n)^\delta=\infty$,
	whenever the series
\equ{series} diverges.  
\end{theorem}

%\smallskip
%\comm{New definitions:}

Clearly  \equ{comm} holds when $f = \const$ or $f = \Id_X$, but not in general. Next we present an alternative approach to upgrading Theorem \ref{thm2.4} to a full measure result, requiring introducing additional assumptions on $(X,\mu,T)$.

\smallskip

Namely, let $\{X_i\}_{i\in  \mathcal{I}}$ be as defined above. We say the partition $\{X_i\}_{i\in \mathcal{I}}$ is \textsl{pseudo-Markov} with respect to $T$ if 
\begin{itemize}
\item for all $i,j\in \mathcal{I}$, $TX_i$ is measurable;
\item  $TX_i\cap X_j\neq \varnothing$ implies $X_j\subset TX_i$  for any $i,j\in  \mathcal{I}$;
%. Also let us say that  $T$ is  \textsl{quasi-bijective} if 
\item there exists   $\tau>0$ such that  $\mu(TX_i)\geq \tau\mu(X)$ for any $i\in  \mathcal{I}$. 
\end{itemize}

\begin{theorem}\label{thm3}
	Let $(X,\mu,T)$, $\psi:\mathbb{N}\to \mathbb{R}$ and $f:X\to X$ be as in Theorem \ref{thm2.4}. Furthermore, assume that $\{X_i\}_{i\in \mathcal{I}}$ is pseudo-Markov.
	% and $T$ is quasi-bijective.  
	Then $\mu\big(R_T^f(\psi)\big)=1$
	%$\sum_{n=1}^\infty \psi(n)^\delta=\infty$,
	whenever the series
\equ{series} diverges.  
\end{theorem}

Examples of systems with pseudo-Markov (in fact, truly Markov) partitions include the Gauss map, the multiplication map $M_b$ where $b\ge 2$ is an integer, and, more generally, conformal expanders described in \S\ref{example}. One can also show that  $\beta$-transformations for some specific $\beta$ admit pseudo-Markov  partitions.
%{Sections 3.1--3.2 of \cite{hussain} together with Remark \ref{comparison} show that in both examples above other assumptions of Theorem \ref{thm3}, namely \eqref{ar}--\eqref{conf}, are also satisfied. In fact, in both cases uniform mixing with exponential rate was first exhibited in \cite{Ph}, together with a quantitative shrinking target property of these systems. }
This is however not true for arbitrary $\beta$. Yet, the twisted recurrence set-up was recently considered in \cite{LVW} for   $T = M_\beta$, where $\beta > 1$ is arbitrary, establishing  the conclusion of  {Theorems \ref{Thm2} and \ref{thm3}} in that case. Namely they prove

\begin{theorem}\label{thmbeta} 
Let $X = [0,1]$, $T = M_\beta$, $\mu$ the $M_\beta$-invariant probability measure on $[0,1]$, and let $\psi:\mathbb{N}\to \mathbb{R}$ and $f:X\to X$ be as in Theorem \ref{thm2.4}. 	% and $T$ is quasi-bijective.  
	Then $\mu\big(R_T^f(\psi)\big)$ is equal to $0$ (resp., $1$) 
	%$\sum_{n=1}^\infty \psi(n)^\delta=\infty$,
	whenever the series
\equ{series} converges (resp., diverges).
%Let $\beta > 1$ and let $T: [0,1]\to [0,1]$ be given by $x\mapsto \beta x\mod{1}$.
	%Let $(X,\mu,T),\psi:\mathbb{N}\to \mathbb{R},f:X\to X$ be as in Theorem \ref{thm2.4}. Furthermore, assume $\{X_i\}_{i\in \mathcal{I}}$ is pseudo-Markov and $T$ is quasi-bijective. 
%	Then if $\sum_{n=1}^\infty \psi(n)%^\delta
%	=\infty$, then $\Leb\left(R_T^f(\psi)\right)=1$. 
\end{theorem}

%\comm{Then need to state a result about full measure. Can do it under two assumptions:
%\begin{itemize}
%\item[(1)] $T$ commutes with $f$;
%\item[(2)] some condition which generalizes $\beta$-transformations.
%\end{itemize}}

\ignore{\comm{Here we need to carefully compare our setting with that of \cite{hussain}. Part of this discussion is below.}

\begin{remark}
In \cite{hussain}, they set each $X_i$ to be an open subset of $\mathbb{R}^n$ and $T|_{X_i}$ to be $C^1$. They assume that there exists $K>0$ such that 
\begin{equation}
	K^{-1}\leq \frac{d(T^mx,T^my)}{d(x,y)||D_x(T^m)||}\leq K \label{eq11}
\end{equation}   
for any $m\in \mathbb{N}$ and $x,y\in J_m\in \mathcal{F}_m$, where $||D_x(T^n)||=\sup_{v\in X}\frac{||D_xT(v)||_2}{||v||_2}$. This is equivalent to our assumption of bounded distortion under their settings. Given (\ref{eq10}), since $||D_x(T^n)||=\sup_{z}\frac{d(T^mx,T^mz)}{d(x,z)}$, (\ref{eq11}) holds. Conversely given (\ref{eq11}), 
\begin{align*}
K^{-2}\leq  \frac{d(T^mx,T^my)/d(x,y)}{d(T^mx,T^mz)/d(x,z)}=\frac{d(T^mx,T^my)}{d(x,y)||D_x(T^m)||}\cdot \frac{d(x,z)||D_x(T^m)||}{d(T^mx,T^mz)}\leq K^2
\end{align*}
so (\ref{eq10}) holds. \comm{Conclusion?} Hence our definitions of $K_J$, bounded distortion and expanding properties match with theirs in the settings where $X_i$ are open subsets of $\mathbb{R}^n$ and $T|_{X_i}$ are $C^1$. Hence in \S \ref{example}, we refer readers to \cite{hussain} for checking that the first two systems satisfy the assumptions.
\end{remark}}

In \S \ref{beta} we show how our methods can be modified to yield an
%We give 
an independent proof of the above theorem.
%\comm{We should say that we also can prove that, maybe state a theorem here.}

\smallskip

We also remark that the paper \cite{DFL} suggests an even more general set-up: there the authors consider a uniformly
Lipschitz function $\Phi:X\times X\to\R$  and under certain assumptions recover zero--one laws for  sets of the form
\begin{equation*}\label{dfl}
\left\{x\in X\left|\begin{aligned}\phi_1(n)\le \Phi (x, T^n x)\le \phi_2(n)   \\ \text{ for infinitely many } n\in \mathbb{N}\ \ \end{aligned}\right.\right\}.
\end{equation*}
Our set-up corresponds to $\phi_1 = 0$, $\phi_2 = \psi$ and $\Phi(x,y) = d\big(f(x),y\big)$. 
It would be interesting to see if the methods of our paper can be applied to the generalized setting of \cite{DFL}.
\smallskip

{The structure of the paper is as follows. In \S\ref{more} we discuss several basic properties of $f$-twisted recurrence sets and some examples of such sets. In \S3 we prove the convergence part of Theorem \ref{thm2.4}. In \S4 we study quasi-independence  properties of the  sequence of measurable sets whose limsup set is %the $f$-twisted recurrence sets
given by \eqref{twisted}. 
%  and give an estimate of the quasi-independence of the sets at the tail of the sequence. 
In \S\S5--6  we consider the divergence case and  complete the proof of Theorems \ref{thm2.4}, \ref{Thm2} and \ref{thm3}. In \S7 we discuss examples of dynamical systems to which our theorems apply. %contains several applications of our main theorem. 
The final section contains a separate discussion of $\beta$-transformations and results in proving Theorem \ref{thmbeta}.} \ignore{\comm{Change after everything else is finalized.}}
%ome direct applications of Theorem \ref{thm2.4}. 

\ignore{With this terminology, Theorem \ref{prt} says that almost every point $x\in X$ is $f$-recurrent when $f$ is the identity map, and 
Theorem \ref{density} asserts that when $T$ is ergodic, $\supp\mu = X$ and  $f \equiv \const$, almost every point $x\in X$ is $f$-recurrent. }

\subsection*{Acknowledgements}
The authors are grateful to Dmitry Dolgopyat, Bassam Fayad, Mumtaz Hussain, Osama Khalil, Bao-Wei Wang and {two anonymous referees} for helpful discussions.

\section{More about $f$-twisted recurrence}\label{more}

We start with several elementary observations concerning sets of $f$-twisted recurrence.

\begin{lemma}\label{piecewise}
	Let   $\psi:\mathbb{N}\to \mathbb{R}^+$ be an arbitrary function, and let $f:X\to X$ be such that  there exist at most countably many %pairwise disjoint 
measurable subsets $Y_i$ of $X$ and  functions $f_i:X\to X$, $i\in\mathcal{I}$, such that $\mu(X\ssm \cup_iY_i) = 0$, 
\begin{equation}\label{pieces}
f|_{Y_i} = f_i\text{ and }\mu\big(R_T^{f_i}(\psi)\big) = 1 \text{ for each }i\in\mathcal{I}.
\end{equation} Then $ \mu\big(R_T^{f}(\psi)\big)  = 1$. 
\end{lemma}

\begin{proof} Indeed, it follows from \eqref{twisted} and \eqref{pieces} that  $$\mu\big(R_T^{f}(\psi) \cap Y_i \big) = \mu\big(R_T^{f_i}(\psi) \cap Y_i \big)  = \mu(Y_i)$$ for each $i\in\mathcal{I}$. \end{proof}

Let us say that a function is \textit{simple} if it   takes at most countably many values.

\begin{corollary}\label{ergsimple}
	Suppose $T$ is ergodic and $\supp\mu$ is dense in $X$. Then $ \mu(R_T^{f})  = 1$ for any simple function $f:X\to X$. 
\end{corollary}

\begin{proof} Immediate from Lemma \ref{piecewise} and \eqref{erg}. \end{proof}

\ignore{\begin{proof}
	Let $f$ be a simple function. Then $f=\sum_{n\in \mathcal{N}}f_n\cdot \chi_{Y_n}$ for a countable collection indexed by $\mathcal{N}$ of constants $f_n$ and disjoint measurable sets $\{Y_n\}_{n\in \mathcal{N}}$ so that $\mu(\bigcup_{n\in \mathcal{N}}Y_n)=1$. Note that 
	\begin{equation*}
		\mu(R_T^{f_n}\cap Y_n)=\mu(Y_n)
	\end{equation*} where $R_T^{f_n}$ is the set of all $f_n$-recurrent points. 
	Hence 
	\begin{equation*}
		\mu(R_T^{f})=\sum_{n\in \mathcal{N}}\mu(R_T^{f_n}\cap Y_n)=\sum_{n\in \mathcal{N}}\mu(Y_n)=1
	\end{equation*}
\end{proof}}

\begin{lemma}\label{approx} Let $(f_n)_{n\in\N}$ be a sequence of functions  $X\to X$  such that  $\mu(R_T^{f_n}) = 1$ for each $n$. Suppose that $f_n\to f$ uniformly on a set of full measure. 
%\comm{(Perhaps this can be relaxed to `almost uniform' convergence: $\forall\,\varepsilon > 0$ $\exists\, Y\subset X$ such that  $\mu(X\ssm Y ) < \varepsilon$ and $f_n\to f$ uniformly on $Y$.)} 
Then $ \mu(R_T^{f})  = 1$. 
\end{lemma}

\begin{proof} Since $\bigcap_nR_T^{f_n}$ has full measure,  for almost every $x\in X$ and each $n\in\N$ one has 
\begin{equation*}
	\liminf_{k\to\infty}d\big(T^kx,f_n(x)\big)=0.
\end{equation*}
	 Fix $\varepsilon > 0$; then there exists $N$ so that for all $n>N$, $d\big(f_n(x),f(x)\big)<\frac{\varepsilon}{2}$ for almost every $x\in X$; on the other hand, for almost every $x\in X$ such that $d\big(f_n(x),f(x)\big)<\frac{\varepsilon}{2}$, $d\big(T^k x,f_n(x)\big)<\frac{\varepsilon}{2}$ for infinitely many $k$. This implies $d\big(T^kx,f(x)\big)<\varepsilon$ for infinitely many $k$.  Since $\varepsilon$ is chosen arbitrarily, we have
%\begin{equation*}
	$\liminf_{k\to\infty}d\big(T^kx,f(x)\big)=0$.
%\end{equation*}
\end{proof}

\begin{corollary}\label{cor1.4}
	Suppose that %$X$ is separable and 
	$T$ is ergodic   and $\supp\mu$ is dense in $X$. Then   $ \mu(R_T^{f})  = 1$ for any Borel-measurable $f:X\to X$.
\end{corollary}  

\begin{proof}  Let $\{x_n\}_{n=1}^\infty$ be a dense subset of $X$.	Let $\varepsilon>0$ and $f:X\to X$ be a Borel-measurable function. Then $\{B(x_n,\varepsilon)\}_{n=1}^\infty$ covers $X$. Define 
 \begin{equation}
 		g_\varepsilon(x)=x_n \text{ where }n=\inf_{m}\{m:f(x)\in B(x_m,\varepsilon)\}
 \end{equation}
 Then $g_\varepsilon$ is simple and $||g_\varepsilon-f||_\infty\leq \varepsilon$. Since $\varepsilon$ is chosen arbitrarily, $f$ is a uniform limit of simple functions. By Corollary \ref{ergsimple} and Lemma \ref{approx}, $x$ is $f$-recurrent for almost every $x\in X$.
\end{proof}

Next, let us observe that the properties of sets $R_T^f(\psi)$ could be strikingly different from the conclusion of Theorem \ref{thm2.4} if the assumptions of that theorem are not imposed. Let us start with the simplest possible non-trivial\footnote{For us ergodic self-maps $T$ of finite sets $X$ will be trivial: indeed, since those are transitive, it easily follows that  $R_T^f(\psi) = X$ for any $f$ and any positive $\psi$.} example of an ergodic dynamical system: an irrational circle rotation \linebreak $X = \R/\Z$, $\mu = $ Lebesgue measure, $T_\alpha(x) = x + \alpha \bmod \Z$ where $\alpha\in\R\ssm \Q$. Then the condition defining the recurrence set $$R_{T_\alpha}(\psi) = \{x : |n\alpha - m| < \psi(n)\text{ for infinitely many }n\in\N\text{ and some }m\in\Z\}$$ is independent of $x$; hence $R_{T_\alpha}(\psi)$ is either $X$ or $\varnothing$, and this dichotomy is different for different $\alpha$. More precisely, Dirichlet's Theorem implies that $R_{T_\alpha}(\psi_1) = X$ for any $\alpha$ (see \equ{psi1} for this notation), and the same is true for $\psi_1$ replaced with $\frac1{\sqrt{5}}\psi_1$, but not with  $c\psi_1$ for $c < \frac1{\sqrt{5}}$. In particular,   $\alpha$ is badly approximable if and only if $R_{T_\alpha}(c\psi_1) = \varnothing$ for some $c > 0$. On the other hand, the theory of continued fractions shows that for any positive non-increasing $\psi$ (decaying arbitrarily fast) there exists $\alpha$ such that  $R_{T_\alpha}(\psi)$ contains $0$ (and hence coincides with $\R/\Z$).

Likewise, studying targets shrinking to $y\in X$ for the above system reduces to inhomogeneous Diophantine approximation:
$$R^y_{T_\alpha}(\psi) = \{x : \dist(n\alpha,y-x) < \psi(n)\text{ for infinitely many }n\in\N\}$$ According to Minkowski's theorem  \cite[Chapter III, Theorem II]{C}, for any irrational $\alpha$ and any $y\in\R/\Z$, the complement of 
%every point not of the form $(y+k\alpha)\bmod 1$ for some $k\in\Z$ belongs to 
$R^y_{T_\alpha}(\frac14\psi_1)$ is at most countable. A  precise zero-one law for sets $R^y_{T_\alpha}( \psi)$ again depends on the Diophantine properties of $\alpha$. For example, 
it is a theorem of Kurzweil \cite{Ku} that   $\alpha$ is badly approximable if and only if the following statement holds: for any non-increasing $\psi$, the set $R^y_{T_\alpha}( \psi)$ is null/conull if $\sum_{k=1}^\infty \psi(k)$ converges/diverges. However, well approximable $\alpha$ come with their own convergence/divergence condition on $\psi$ guaranteeing  that $R^y_{T_\alpha}( \psi)$ is null or conull; see  \cite{FK} for the most general statement.

Clearly the set-up of $f$-twisted recurrence can be similarly and straightforwardly   restated in a Diophantine approximation language:  $$R^f_{T_\alpha}(\psi) = \left\{x\in X: \dist\big(n\alpha,f(x)-x\big) < \psi(n)\text{ for infinitely many }n\in\N\right\}.$$
Thus if $f(x) = x + \beta\bmod \Z$ for a fixed $\beta$, then $R_{T_\alpha}(\psi)$ is either $X$ or $\varnothing$; alternatively, if the pushforward of Lebesgue measure by the map $x\mapsto f(x) - x$ is absolutely continuous with respect to Lebesgue, then the zero/one law for the sets $R^f_{T_\alpha}(\psi)$ depends on the Diophantine properties of $\alpha$ as described in \cite{FK}.

The situation is even trickier if one considers irrational rotations of higher-dimensional tori. Namely, if we let $X = \R^d/\Z^d$ and $\mu = $ Lebesgue measure, then it is shown in \cite{GP} that for any (arbitrarily slowly decaying) non-increasing function $\psi$ with $\lim_{t\to \infty}\psi(t) = 0$ there exists an ergodic translation $T_\alpha:x\mapsto x + \alpha \bmod \Z^d$ such that $\mu\big(R^y_{T_\alpha}(\psi)\big) = 0$ for any $y\in X$. Moreover, by suitably reparametrizing the aforementioned example one can construct a smooth  mixing transformation on the three dimensional torus  with the same property. Thus some conditions on the speed of mixing is crucial for a zero-one law as in Theorem \ref{thm2.4}.

%\vfil\eject
%On the other hand, for any $c < \frac14$ there exists $\alpha\notin\Q$   such that for any $y\in\R/\Z$ and any $x\notin y + \alpha\Z$ such that 

\section{{The convergence part}} \label{quasiind}% of $\{A_n\}_n$}
In the next two sections we prove Theorem \ref{thm2.4}, thereby  assuming that $(X,\mu,T)$ satisfies conditions \eqref{ar}--\eqref{conf} and fixing  $\psi:\mathbb{N}\to \mathbb{R}^+$ with $\lim_{n\to \infty}\psi(n)=0$.  
Similarly to \eqref{recsets} and \eqref{shrsets}, for an arbitrary $f:X\to X$  let us define
\begin{equation}\label{twsets}
	A_n = A_T^f(n,\psi):=\left\{x\in X:d\big(T^n x,f(x)\big)<\psi(n)\right\}.
\end{equation}
%For fixed $T$, $f$ and $\psi$ we will simply denote the above sets by $A_n$. 
Clearly $R_T^f(\psi) = \limsup A_n$.

%Recall that in the shrinking target case, where $f$ was a constant function, that is, $f(x)=y$ for all $x\in X$, we could write $R_T^f(\psi)=\limsup_n B_n$, where $B_n:=T^{-n}B(x_0,\psi(n))$. However, 

\smallskip
Unlike the shrinking target case, corresponding to constant functions $f$, the sets $A_n$ cannot be expressed in the form   $T^{-n}B_n$ for some balls $B_n$. Our strategy is to consider the intersection of $A_n$ with $f^{-1}B(x_0,r)$, where $x_0\in X$ and $r > 0$, and 
%to the preimage under $f$ of a small ball, and 
approximate this intersection by the preimages  of some balls under $T$. \par 

%\eqref{ar} only works for $r < r_0$. So to apply \eqref{ar}, let 

\begin{lemma}\label{lemma1}
	%Let $B=B(x_0,r)$ be a ball centered at $x_0\in X$ and radius $r>0$. Then f
	For any $x_0\in X$, any $r>0$ and any subset $E$ of $f^{-1}B(x_0,r)$, 
	\begin{equation}\label{eq1}
		E\cap A_n \subset E\cap T^{-n}B\big(x_0,\psi(n)+r\big).
	\end{equation}
	Furthermore, if $r <  \psi(n)$, then 
	\begin{equation}\label{lem1-2}
	E\cap T^{-n}B\big(x_0,\psi(n)-r\big)\subset 	 E\cap A_n.
	\end{equation}
\end{lemma}
\begin{proof}
	Fix a point $x\in E\cap A_n$. Then $$d\big(f(x),x_0\big)<r\text{ and }d\big(T^nx,f(x)\big)<\psi(n),$$ which implies that
	\begin{equation*}
		d(T^nx,x_0) < d\big(T^nx,f(x)\big)+d\big(f(x),x_0\big)<\psi(n)+r.
	\end{equation*}
	Hence $E\cap A_n\subset  E\cap T^{-n}B\big(x_0,\psi(n)+r\big)$. \par 
	On the other hand, fix $x\in E\cap T^{-n}B\big(x_0,\psi(n)-r\big)$. Then $d(f(x),x_0)<r$ and $d(T^nx,x_0)<\psi(n)-r$. Hence 
	\begin{equation*}
		d\big(T^nx,f(x)\big)\leq d(T^nx,x_0)+d\big(x_0,f(x)\big)<\psi(n),
	\end{equation*}
	thus $E\cap T^{-n}B\big(x_0,\psi(n)-r\big)\subset E\cap A_n$. 
\end{proof}

%\begin{remark}
%This technique works for any function $f$; Lipshitz condition is not required for Lemma \ref{lemma1} to hold. 
%\end{remark}

%\begin{corollary}
%	Take $B=B(x_0,\varepsilon\psi(n))$ with $0<\varepsilon<1$, 
%	\begin{equation}\label{eq1}
%		f^{-1}B\cap T^{-n}B(x_0,(1-\varepsilon)\psi(n))\subset f^{-1}B\cap A_n\subset f^{-1}B\cap T^{-n}(x_0,3\psi(n)/2)f
%	\end{equation}
%\end{corollary}

%\begin{proof}
%	We take $B=B(x_0,\varepsilon\psi(n))$ and then apply Lemma \ref{lemma1}. 
%\end{proof}

Choose $n_0\in \mathbb{N}$ such that  $5\psi(n) < r_0$ for all $n > n_0$, where $r_0$ is as in \eqref{ar};  the next several statements in this section will be proved for $n > n_0$.

{\begin{lemma}\label{lemma3}
	Let %$0\leq \varepsilon \leq \frac{1}{2}$ and 
	$B=B\big(x_0,  \psi(n)/2\big)$ for some $x_0\in X$ and $n>n_0$. 
	%Assume the system satisfies assumptions 1 and 2. 
	Then for any open ball $E$ contained in $f^{-1}B$, 
	\begin{align*}
		2^{-\delta} \big(\eta_1\mu(E)-\eta_2a_n\big) \psi(n)^\delta \leq \mu(E\cap A_n)\leq \eta_2(3/2)^\delta\big(\mu(E)+a_n\big)  \psi(n)^\delta,
	\end{align*}
	%where $C_1=\eta_1  2^{-\delta}$, $C_2=\eta_22^{-\delta}$, $C_3=\eta_2(3/2)^\delta$, %$C_4=\eta_2(3/2)^\delta$, 
	 with $\delta,\eta_1,\eta_2$  as in \eqref{ar} and $(a_n)_{n\in\N}$ as in \eqref{em}. 
\end{lemma}

\begin{proof}
	Let $r = \psi(n)/2$, and let $E$ be an open ball contained in $f^{-1}B(x_0,r)$. 
Combining (\ref{eq1})   with \eqref{em},
%with equation with the mixing condition (Assumption 2), 
we get
\begin{align*}
		\mu(E\cap A_n)&\geq \mu\Big(E\cap T^{-n}B\big(x_0,\psi(n) - r\big)\Big)\\
		&\geq \mu(E)\mu\big(T^{-n}B(x_0,\psi(n) - r)\big)
		-a_n\mu\big(E\cap T^{-n}B(x_0,\psi(n) - r)\big)\\
		&= \mu(E)\mu\Big(T^{-n}B\big(x_0,\psi(n)/2\big)\Big)
		-a_n\mu\Big(  T^{-n}B\big(x_0,\psi(n)/2\big)\Big)\\
		&= \mu(E)\mu\Big(B\big(x_0,\psi(n)/2\big)\Big)
		-a_n\mu\Big(  B\big(x_0,\psi(n)/2\big)\Big)\\
		&\underset{\eqref{ar}}\geq \big(\eta_1\mu(E)-\eta_2a_n\big)2^{-\delta}\psi(n)^\delta
			\end{align*}
			and
\begin{align*}
		\mu(E\cap A_n)&\leq \mu\Big(E\cap T^{-n}B\big(x_0,\psi(n) + r\big)\Big)\\
		&\leq \mu(E)\mu\big(T^{-n}B(x_0,\psi(n) + r)\big)
		+a_n\mu\big(E\cap T^{-n}B(x_0,\psi(n) + r)\big)\\
		&= \mu(E)\mu\Big(T^{-n}B\big(x_0,3\psi(n)/2\big)\Big)
		+a_n\mu\Big(  T^{-n}B\big(x_0,3\psi(n)/2\big)\Big)\\
		&\leq \mu(E)\mu\Big(B\big(x_0,3\psi(n)/2\big)\Big)
		+a_n\mu\Big( B\big(x_0,3\psi(n)/2\big)\Big)\\
		&\underset{\eqref{ar}}\leq \big(\mu(E)+a_n\big)\eta_2(3/2)^\delta\psi(n)^\delta,
			\end{align*}
%\begin{align*}
%		\mu(E\cap A_n)&\leq \mu(E\cap T^{-n}B(x_0,3\psi(n)/2))\\
%		&\leq \mu(E)\mu(T^{-n}B(x_0,3\psi(n)/2))
%		+a_n\mu(E\cap T^{-n}B(x_0,3\psi(n)/2))\\
%		&\leq \mu(E)\eta_2(1+\varepsilon)^\delta\psi(n)^\delta+\eta_2a_n(1+\varepsilon)^\delta\psi(n)^\delta
%	\end{align*}
establishing the claim. \end{proof}}
%\comm{A correction: \eqref{ar} only works for $r < r_0$, we need to take it into account. Thus this lemma will only hold for large enough $n$ and we will need to use the condition $\lim_{n\to \infty}\psi(n)=0$. Maybe you can define $n_0$ such that $5\psi(n) < r_0$ for $n > n_0$, and prove all the statements in this section for $n > n_0$.}

To prove Theorems \ref{thm2.4}--\ref{thmbeta}, in view of Lemma \ref{piecewise} it is enough to assume that $f$ is Lipschitz.
Thus for the rest of the paper we let $f:X\to X$ be a $p$-Lipschitz function for some $p>0$. 
 
\smallskip

The next lemma estimates the measure of the sets $A_n$. 

%\comm{I kept two versions of Lemma 3.3, so please remove the one you don't like.}

\begin{lemma}\label{lemma4}
	For $n>n_0$, 
	\begin{align*}
		\eta_2^{-1}\eta_1^2 10^{-\delta}\psi(n)^\delta-{(p/5)^{\delta}}a_n&\leq \mu(A_n)%\\&
		\leq {\eta_1^{-1}\eta_2(3/2)^\delta\big(\eta_25^{\delta}\psi(n)^\delta}+%\eta_1^{-1}{(2p)^{\delta}}
		%C_4
		(2p)^{\delta}a_n\big).\end{align*}
		%where $C_1,C_2,C_3$ are as in Lemma \ref{lemma3}. %{for $\varepsilon=\frac{1}{10}$}. 
\end{lemma}

\begin{proof}
	Take $x\in X$, $y\in f^{-1}\{x\}$  % {Let $0<\varepsilon\leq \frac{1}{2}$.} \comm{(replace $\varepsilon$ by $1/2$) Here $\varepsilon$ will be cancelled out in the later computations so now I think maybe using $\varepsilon$ is clearer in view of Lemma 3.2.} 
	and $z\in B\left(y,\frac{\psi(n)}{2p}\right)$. Then by the $p$-Lipshitz condition, $d\big(x,f(z)\big)\leq pd(y,z)<\psi(n)/2$. Thus $$B\big(y, \psi(n)/p\big)\subset f^{-1}B\big(x,\psi(n)/2\big).$$
	We have an open covering
	\begin{equation*}
		\left\{B\big(y, \psi(n)/p\big):y\in X\right\}
	\end{equation*} with each $B\left(y,\frac{\psi(n)}{2p}\right)\subset f^{-1}B\big(x,\psi(n)/2\big)$ for some $x\in X$. \par
	By Vitali's covering theorem ($5r$-covering lemma), we can find countably many disjoint balls $\left\{B\left(y_j,\frac{\psi(n)}{2p}\right)\right\}_{j\in \mathcal{J}}$ such that 
\begin{equation}\label{eq3}
	%\bigcup_{j\in \mathcal{J}}B\big(y_j,\frac{\psi(n)}{2p}\big)\subset
	 X\subset \bigcup_{j\in \mathcal{J}}B\left(y_j,\frac{5\psi(n)}{2p}\right).
\end{equation}
By the %first inclusion of (\ref{eq3}) and 
disjointness of $\left\{B\left(y_j,\frac{\psi(n)}{2p}\right)\right\}_{j\in \mathcal{J}}$, we have 
\begin{equation*}
	\sum_{j\in \mathcal{J}}\eta_1\left(\frac{\psi(n)}{2p}\right)^\delta\leq \sum_{j\in \mathcal{J}}\mu\left(B\Big(y_j,\frac{\psi(n)}{2p}\Big)\right)\leq \mu(X)=1.
\end{equation*} 
Hence $|\mathcal{J}|\leq \eta_1^{-1}\left(\frac{2p}{\psi(n)}\right)^\delta$. On the other hand, by %the second inclusion of 
(\ref{eq3})  we have 
\begin{equation*}
\sum_{j\in \mathcal{J}}\eta_2\left(\frac{5\psi(n)}{2p}\right)^\delta\geq \sum_{j\in\mathcal{J}}\mu\left(B\Big(y_j,\frac{5\psi(n)}{2p}\Big)\right)\geq \mu(X)=1;
\end{equation*}
hence $|\mathcal{J}|\geq \eta_2^{-1}\left(\frac{2p}{5\psi(n)}\right)^{\delta}$. \par 
%Note that 
%\begin{equation*}
%	A_n= \bigcup_{j\in \mathcal{J}}B(y_j,5\psi(n)/2p)\cap A_n
%\end{equation*}
By Lemma \ref{lemma3}, % and now taking {$\varepsilon=\frac{1}{10}$}, 
since for each $j$ we have  $B\left(y_j,\frac{\psi(n)}{2p}\right)\subset f^{-1}B\big(x_j,\psi(n)/2\big)$ and $B\left(y_j,\frac{5\psi(n)}{2p}\right)\subset f^{-1}B\big(x_j,5\psi(n)/2\big)$, it follows that
\begin{align*}
\mu(A_n)\leq & 	\sum_{j\in J} \mu\left(B\Big(y_j,\frac{5\psi(n)}{2p}\Big)\cap A_n\right)\\
\leq & \sum_{j\in J}\eta_2(3/2)^\delta\left[\mu\left(B\Big(y_j,\frac{5\psi(n)}{2p}\Big)\right)+a_n\right] \psi(n)^\delta\\
\leq & \ \eta_1^{-1}\left(\frac{2p}{\psi(n)}\right)^\delta \eta_2(3/2)^\delta\left[ \Big(\frac{5\psi(n)}{2p}\Big)^\delta \eta_2+a_n\right]\psi(n)^\delta\\
{=} &\  {\eta_1^{-1}\eta_2(3/2)^\delta\big(\eta_25^{\delta}\psi(n)^\delta}+%\eta_1^{-1}{(2p)^{\delta}}
		%C_4
		(2p)^{\delta}a_n\big)  \end{align*}
 and
\begin{align*}\mu(A_n) \geq & \sum_{j\in\mathcal{J}}\mu\left(B\Big(y_j,\frac{\psi(n)}{2p}\Big)\cap A_n\right)\\
\geq &\ \sum_{j\in J} \left[\eta_1  2^{-\delta}\mu\left(B\Big(y_j,\frac{\psi(n)}{2p}\Big)\right)-\eta_22^{-\delta}a_n \right]\psi(n)^\delta\\
\geq &\  \eta_2^{-1}\left(\frac{2p}{5\psi(n)}\right)^{\delta}\left[\eta_1  2^{-\delta}  \eta_1\Big(\frac{\psi(n)}{2p}\Big)^\delta -\eta_22^{-\delta}a_n\right]\psi(n)^\delta\\
{=}  &\  \eta_2^{-1}\eta_1^2  10^{-\delta}\psi(n)^\delta-\eta_2^{-1}{(2p/5)^{\delta}}\eta_22^{-\delta}a_n,
\end{align*}
%{where $C_1,C_2,C_3,C_4$ are defined in Lemma \ref{lemma3} with $\varepsilon=\frac{1}{10}$}.
finishing the proof of the lemma.\end{proof}

\begin{proposition}\label{propm}
	\begin{equation}\label{p17}
		\sum_{n=1}^\infty \psi(n)^\delta=\infty \iff \sum_{n=1}^\infty \mu(A_n)=\infty
	\end{equation}
\end{proposition}

\begin{proof}
	By Lemma \ref{lemma4} we know that 
	\begin{align*}
		&%\eta_2^{-1}\left(\eta_1C_15^{-\delta}\sum_{n>n_0}%^\infty
		% \psi(n)^\delta-{(2p/5)^{\delta}}C_2\sum_{n>n_0}
		 %^\infty 
		% a_n	\right) 
		\eta_2^{-1}\eta_1^2 10^{-\delta}\sum_{n>n_0}\psi(n)^\delta-{(p/5)^{\delta}}\sum_{n>n_0}a_n\\
		\leq & 	\sum_{n>n_0}%^\infty
		 \mu(A_n) \\
		\leq & \ \eta_1^{-1}\eta_2(3/2)^\delta\left(\eta_25^{\delta}\sum_{n>n_0}\psi(n)^\delta+%\eta_1^{-1}{(2p)^{\delta}}
		%C_4
		(2p)^{\delta}\sum_{n>n_0}a_n\right).
	\end{align*} Since $\{a_n\}$ is summable, (\ref{p17}) holds. 
\end{proof} \par 

\begin{remark} \rm
{Note that  Proposition \ref{propm} and the Borel--Cantelli Lemma immediately imply the convergence case of Theorem \ref{thm2.4}: if $\sum_{n=1}^\infty \psi(n)<\infty$, then} %$\sum_{n=1}^\infty \mu(A_n)<\infty$, so by , 
$\mu\big(R^f_T(\psi)\big)=\mu(\limsup_n A_n)=0$. Note also that or this conclusion one only needs the first two conditions of Theorem \ref{thm2.4}, that is,  \eqref{ar} and \eqref{em};  the remaining conditions \eqref{eq10}--\eqref{conf} will be used in the proof of the divergence case.
\end{remark}

\section{A quasi-independence estimate }\label{qind}

Now let us make use of assumptions \eqref{eq10}--\eqref{conf}.
The following   lemma %\ref{lemma6} and \ref{lemma3.7} 
was stated and used in \cite{hussain}; we prove it here since our set-up is slightly different. 
\begin{lemma}\label{lemma6}
	For $m\in\N$, $J$ a cylinder in $\mathcal{F}_m$ and for any open set $U$ contained in $J$, $\mu(T^m U)\asymp K_{J}^\delta \mu(U)$. 
\end{lemma}

\begin{proof}
	By %Assumptions 1 and 6, 
	\eqref{ar} and \eqref{conf}, we know  that for all open balls $B\subset J$ with  {radius smaller than $r_0$}, it holds that $\mu(T^m B)\asymp K_{J}^\delta \mu(B)$.	 Let $U\subset J$ be an open subset. Consider the cover 
	 \begin{equation*}
	 	\mathcal{S}=\big\{B(x,r):x\in U, \ B(x,5r)\subset U,\  r<r_0\big\}
	 \end{equation*}
	 of $U$. By Vitali's covering theorem, $\mathcal{S}$ has a countable sub-collection $\mathcal{B}$ of disjoint balls so that 
	 \begin{equation*}
	 	\bigcup_{B(x,r)\in \mathcal{B}}B(x,r)\subset U \subset \bigcup_{B(x,r)\in \mathcal{B}}B(x,5r).
	 \end{equation*}
	 Since $T^m$ is injective on $J$,
	 \begin{equation*}
	 	\bigcup_{B(x,r)\in \mathcal{B}}T^mB(x,r)\subset T^mU \subset \bigcup_{B(x,r)\in \mathcal{B}}T^mB(x,5r).
	 \end{equation*}
	 Hence 
	 \begin{equation*}
	 	K_J^\delta \sum_{B(x,r)\in \mathcal{B}}\mu\big(B(x,r)\big)\asymp \mu(T^mU).
	 \end{equation*}
	 On the other hand, $\mu(U)\asymp \sum_{B(x,r)\in \mathcal{B}}\mu\big(B(x,r)\big)$, and the lemma is proved. 
\end{proof}

\ignore{\comm{Maybe we should add a remark here somehow addressing the reviewer's comment.}}

%Some of the following statements only hold for large enough indices, which we specify here. 
	
{Now recall that we were working with the sets $A_n$ defined in \eqref{twsets}. The next lemma shows that the intersection of a cylinder of high enough level with $A_n$ is contained in a small ball. Namely,} let $m_0 \geq n_0$ be such that $K_J>\max\left\{\frac{K_1\diameter (X)}{r_0}, 2p\right\}$ for all $m>m_0$ and $J\in \mathcal{F}_m$ (which is possible in view of \eqref{toinfty}).

\begin{lemma}\label{lem3.6}
 {For $m>m_0$}, for every cylinder $J\in \mathcal{F}_m$ and any $z\in J\cap A_m$ there exists a ball of radius 
\begin{equation}\label{starR}r=\frac{2\psi(m)}{K_{J}-p}, 	\end{equation}	
say $B(z,r)$, such that 
	\begin{equation}\label{starJ}
		J\cap A_m\subset B(z,r)\cap J.
	\end{equation}	
\end{lemma}

\begin{proof}
	Choose any $x,z\in J\cap A_m$. %Let $x\in J\cap A_m$. 
	Since $J\in \mathcal{F}_m$, in view of \eqref{eq10} we have	\begin{equation*}
	d(T^mx,T^mz)K_{J}^{-1}\geq d(x,z)	.
	\end{equation*}
On the other hand,
\begin{align*}
	d(T^mx,T^mz)&\leq d\big(T^mx,f(x)\big)+d\big(f(x),f(z)\big)+d\big(T^mz,f(z)\big)\\
	&\leq  2\psi(m)+pd(x,z).
\end{align*}
Then $K_{J}d(x,z)<2\psi(m)+pd(x,z)$, i.e.\ $d(x,z)<\frac{2\psi(m)}{K_{J}-p}$. 
	\end{proof}

Now let us prove a quasi-independence property of the sequence $\{A_n\}_n$. For any $m\in\N$ and $J\in \mathcal{F}_m$ define 
\begin{align}
	J^*:=B(z,r)\cap J,\label{jstar}
\end{align}
where $r$ and $z$ are defined in  (\ref{starR}) and (\ref{starJ}). 

\begin{lemma}\label{localqilemma}
For all $n>m>m_0$ and for each $J\in \mathcal{F}_m$, 
\begin{equation*}
	\mu(J\cap A_m\cap A_n)\lesssim K_{J}^{-\delta}\left[\psi(m)^\delta\psi(n)^\delta+a_{n-m}\psi(n)^\delta+a_n\psi(m)^\delta\right].
\end{equation*}
\end{lemma}

\begin{proof}
	Let $J\in \mathcal{F}_m$. %We have showed 
	By Lemma \ref{lem3.6},
	\begin{equation*}
		J\cap A_m\cap A_n \subset J^*\cap A_n
	\end{equation*} where $J^*$ is defined in \eqref{jstar}. Now let us estimate $\mu(J^*\cap A_n)$. \par 
	\smallskip
	\underline{Case (i)}: $pr=\frac{2p\psi(m)}{K_{J}-p}\leq \psi(n)$.  
	
	\smallskip
{Note  that for all $x\in B(z,r)$ we have $d\big(f(x),f(z)\big)<pd(x,z)<pr$, therefore $B(z,r)\subset f^{-1}B\big(f(z),pr\big)$}. Thus $J^*$ is a subset of $f^{-1}B\big(f(z),pr\big)$, and we can apply Lemma \ref{lemma1} to $J^*$ and obtain 
	\begin{equation}\label{eq9}
		J^*\cap A_m \subset J^*\cap T^{-n}B\big(f(z),\psi(n)+pr\big)\subset  J^*\cap T^{-n}B\big(f(z),2\psi(n)\big).
	\end{equation}
	Then apply Lemma \ref{lemma6} to $J^*\cap T^{-n}B\big(f(z),2\psi(n)\big)$, getting $$
	\begin{aligned}
	\mu\Big(J^*\cap T^{-n}B\big(f(z),2\psi(n)\big)\Big)&\lesssim K_{J}^{-\delta}\Big(T^mJ^*\cap T^{-(n-m)}B\big(f(z),2\psi(n)\big)\Big)\\
	%$$ Since $J^* \subset B(z,r)$, $T^m(J^*)\subset T^mB(z,r)$. So far  
%	\begin{equation*}
%		\mu(J^*\cap A_n)\lesssim 
&\le K_{J}^{-\delta}\mu\Big(T^mB(z,r)\cap T^{-(n-m)}B\big(f(z),2\psi(n)\big)\Big).
\end{aligned}	$$
{Since $m>m_0$, we have $\inf_{J\in\mathcal F_m}K_{J}>2p$. Then by the conformality assumption (1.12), we have}
 \begin{align*}
 T^mB(z,r)&=T^m B\left(z,\frac{2\psi(m)}{K_{J}-p}\right) \\
 &\subset B\left(T^m z, K_2 K_{J}\frac{2\psi(m)}{K_{J}-p}\right)\subset B\big(T^m z,K_24\psi(m)\big),
 \end{align*}  where $K_2$ is defined in (1.12). 
Thus 
$$\mu(J^*\cap A_n)\lesssim K_{J}^{-\delta}\mu\Big( B\big(T^mz,4K_2\psi(m)\big)\cap T^{-(n-m)}B\big(f(z),2\psi(n)\big)\Big).$$

  By the mixing property %(Assumption 2), 
  (1.8), 
		\begin{align*}
		&\mu(J^*\cap A_n)\lesssim K_{J}^{-\delta}\mu\Big(T^m B\big(T^mz,4K_2\psi(m)\big)\cap T^{-(n-m)}B\big(f(z),2\psi(n)\big)\Big)\\
		\lesssim\  &K_{J}^{-\delta} \left[\mu \Big(B\big(T^mz,4K_2\psi(m)\big)\Big)  \mu \Big(B\big(z,2\psi(n)\big)\Big)+ a_{n-m}\mu \Big(B\big(z,2\psi(n)\big)\Big)\right]\\
		\lesssim\  &K_{J}^{-\delta} \left[\psi(m)^\delta\psi(n)^\delta+a_{n-m}\psi(n)^\delta\right].
	\end{align*}\par 
	
	\smallskip
	\underline{Case (ii)}: $pr=\frac{2p\psi(m)}{K_{J}-p}> \psi(n)$.
	
	\smallskip
	
	 We replace the ball $B\big(f(z),pr\big)$ by a collection of balls of radius $\psi(n)$. Choose a maximal $\frac{3}{2}\psi(n)$-separated set in $B\big(f(z),pr\big)$, denoted by $\{z_i\}_{1\leq i\leq N_{m,n}}$. Then 
	\begin{align*}
		B\big(f(z),pr\big)\subset \bigcup_{i=1}^{N_{m,n}} B\big(z_i,\psi(n)\big)\subset B\big(f(z),2pr\big).
	\end{align*}
	%\comm{Are $B\big(z_i,\psi(n)\big)$ disjoint? maybe you need the points to be $2\psi(n)$-separated? (\checkmark)}
	Since $\mu$ is Ahlfors regular and $m>n_0$, %\comm{we also need $m$ large enough so that $r < r_0$ (\checkmark)}
	\begin{align*}
	N_{m,n}\asymp \left(\frac{p\psi(m)}{(K_{J}-p)\psi(n)}\right)^\delta .
	\end{align*}
	Since $B\big(z_i,\psi(n)\big)\subset f^{-1}B\big(f(z_i),p\psi(n)\big)$, we can apply Lemma 3.1 to each ball $B\big(z_i,\psi(n)\big)$ with $1\leq i\leq N_{m,n}$ and obtain
	%\comm{I don't understand the first inequality below. (\ref{eq9}) is valid under the assumption of part (i), and now we have part (ii). Why can we use it?}
	\begin{align*}
		\mu\Big(B\big(z_i,\psi(n)\big)\cap A_n\Big)
		\underset{\eqref{eq1}}\leq & \mu\Big(B\big(z_i,\psi(n)\big)\cap T^{-n} B\big(z_i,(1+p)\psi(n)\big)\Big)  \quad \begin{aligned}
		\end{aligned}\\
		\underset{\eqref{em}}\leq &\mu\Big(B\big(z_i,\psi(n)\big)\Big)\mu\Big(B\big(z_i,(1+p)\psi(n)\big)\Big)+a_n\mu\Big(B\big(z_i,(1+p)\psi(n)\big)\Big)\\
		\underset{\eqref{ar}}\lesssim & \left[\psi(n)^\delta+a_n\right] \psi(n)^\delta.
	\end{align*}
	Now summing over $1\leq i\leq N_{m,n}$, we have 
	\begin{align*}
		\mu(J^*\cap A_n)\leq &\sum_{i=1}^{N_{m,n}}\mu\Big(B\big(z_i,\psi(n)\big)\cap A_n\Big)\\
		\lesssim& \left(\frac{p\psi(m)}{K_{J}-p}\right)^\delta \left[\psi(n)^\delta+a_n\right]\\
		\lesssim & \ K_J^{-\delta}[\psi(m)^\delta\psi(n)^\delta+a_n\psi(m)^\delta ].
	\end{align*}
Combining the two cases we obtain the desired conclusion. \end{proof}

\begin{proposition}\label{globalqi}
	 {For $n>m>m_0$,}
	 \begin{equation*}
	 	\mu(A_m\cap A_n)\lesssim \psi (m)^\delta\psi(n)^\delta+a_{n-m}\psi(n)^\delta+a_n\psi(m)^\delta.
	 \end{equation*}
%where $(a_n)_{n\in\N}$ is as in \eqref{em}.
 \end{proposition}
 
 \begin{proof}
 	This follows directly from the previous lemma and \eqref{qi}, since 
 	\begin{align*}
 			\mu(A_m\cap A_n)\lesssim & \sum_{J\in \mathcal{F}_m}K_{J}^{-\delta}\left[\psi(m)^\delta\psi(n)^\delta+a_{n-m}\psi(n)^\delta+a_n\psi(m)^\delta\right]\\
 			\underset{\eqref{qi}}\lesssim & \psi(m)^\delta\psi(n)^\delta+a_{n-m}\psi(n)^\delta+a_n\psi(m)^\delta.
 	\end{align*}
 \end{proof}

\section{%The divergence part %: positive measure of $R_T^f(\psi)$ \comm{(?)}
Proof of Theorems \ref{thm2.4} and %Theorem 
\ref{Thm2} }\label{pf}

To prove the divergence case of Theorem \ref{thm2.4}, let us recall
%  which states 
 \begin{lemma}[Chung--Erd\"{o}s inequality, \cite{CE}] \label{cei}
	Let $(X,\mu)$ be a probability space, and let $\{E_n\}_n$ be a sequence of events such that $\sum_{n=1}^\infty \mu(E_n)=\infty$. Then 
	\begin{equation*}
		\mu\left(\limsup_{n\to\infty} E_n\right)\geq \limsup_{N\to \infty} \frac{\left(\sum_{n=1}^N \mu(E_n)\right)^2}{\sum_{n,m=1}^N \mu(E_n\cap E_m)}. 
	\end{equation*}
\end{lemma}

The next lemma is based on  the above inequality.
%Now we have enough ingredients to prove that $R_T^f(\psi)$ has  positive measure, using the the Chung-Erd\"{o}s inequality. We start with the following abstract lemma.

\begin{lemma}\label{posmeas}
	Let %$(\Omega,\mathcal{A},\nu)$ 
	$(X,\mu)$ be a probability space. Let $(E_k)_k$ be a sequence of measurable subsets of $X$, and let $(a_n)_n,(b_n)_n$ be sequences of positive %real 
	numbers such that $ \sum_{n=1}^\infty a_n < \infty$ and $\sum_{n=1}^\infty b_n = \infty$.
	% is not summable and $\sum a_k=S$, and 
	Assume that for some %positive constants 
	$s_1,s_2,s_3>0$ %and $m_0\in\N$ 
	it holds that
	$$
		  s_1(b_n-a_n)\leq \mu(E_n) \leq  s_3(b_n+a_n)
		\quad \text{for all }n \in\N$$
		and $$
		\mu(E_m\cap E_n) \leq  s_2 (b_mb_n+a_{n-m}b_n+a_n b_m) \quad \text{for all }  m<n.
	$$
	Then $\mu(\limsup_{n\to\infty}E_n)\geq %\frac{s_1^2}{2[s_3+2s_2(1+2S)]}
	\frac{s_1^2}{2s_2}$.\end{lemma}
	
	\begin{proof}
		%Define $f_n=\sum_{k=1}^n \chi_{E_k}$. 
		Let us denote $\sum_{k=1}^\infty a_k=S$. 
		Choose $\varepsilon > 0$. On the one hand, 
	%\begin{align*}
		%\left(
	%	\int f_n\,d\mu%\right)^2
	%	&
	%	=%\left[
	$$	\sum_{n=1}^N \mu(E_n)%\right]^2
		%%\\
		%&
		\geq %\left[
		 s_1\sum_{n=1}^N (b_n-a_n)%\right]^2
		%\\
		%&
		\geq %\left[
		 s_1\left(\sum_{n=1}^N b_n-S\right)%\right]^2
		\\
		\geq %\left[
		 (s_1 -\varepsilon)\sum_{n=1}^N b_n%\right]^2 
	$$	
	when $N$ is sufficiently large,  because $\sum_{n=1}^\infty  b_n=\infty$.
	On the other hand, 
	\begin{align*}
		\sum_{m,n=1}^N \mu(E_m\cap E_n) &=\sum_{n=1}^N \mu(E_n)+
		2\sum_{1\le m<n\leq N}\mu(E_m\cap E_n)\\
		&\leq   s_3\sum_{n=1}^N (b_n+a_n)+
		2 s_2\sum_{m=1}^N\sum_{n=m+1}^N(b_mb_n+a_{n-m}b_n+a_n b_m)\\
		&\leq    s_3\sum_{n=1}^N b_n+s_3S+
		2s_2\sum_{m=1}^N\sum_{n=m+1}^Nb_mb_n+2s_2S\sum_{n=1}^nb_n+2s_2S\sum_{m=1}^n b_m\\
		&\leq   2s_2\left(\sum_{n=1}^N b_n\right)^2 + (s_3+4s_2S)\sum_{n=1}^N b_n+s_3S\leq   (2s_2 + \varepsilon) \left(\sum_{n=1}^N b_n\right)^2  %\text{when }n \text{ is sufficiently large}.
	\end{align*}
		when $N$ is sufficiently large,  again because $\sum_{n=1}^\infty  b_n=\infty$.
	Hence by Lemma~\ref{cei}
$$		\mu\left(\limsup_{n\to\infty} E_n\right)\geq \limsup_{N\to \infty} \frac{\left((s_1 -\varepsilon)\sum_{n=1}^N b_n\right)^2}{ (2s_2 + \varepsilon) \left(\sum_{n=1}^N b_n\right)^2} = 
\frac{(s_1 -\varepsilon)^2}{ 2s_2 + \varepsilon},$$
%	\begin{align*}
%		\mu(\limsup_{k} E_k)&\geq \limsup_n \frac{(\int f_n\, d\mu)^2}{\int f_n^2\,d\mu}
%	\geq %\limsup_n
%	 \frac{(s_4s_1)^2}{2s_4[s_3+2s_2(1+2S)]}.
%	\end{align*}
and, since $\varepsilon$ was arbitrary, the conclusion follows.\end{proof}
	
	\smallskip

\begin{proof}[Proof of Theorem \ref{thm2.4}, the divergence part]
	Take $E_n=A_{m_0+n}$, where $m_0$ is chosen prior to Lemma \ref{lem3.6}. By Lemma \ref{lemma4}, there exist %positive constants 
	$s_1,s_3> 0$ %and $s_3$ so that 
	such that
	\begin{equation*}
		s_1\big(\psi(m_0+n)^\delta-a_{m_0+n}\big) \leq \mu(E_n) \leq	s_3\big(\psi(m_0+n)^\delta+a_{m_0+n}\big)%\leq \mu(A_n)\leq s_3\big(\psi(n)^\delta+a_n\big).
	\end{equation*}
	for any $n\in\N$. Also by Proposition  
\ref{globalqi} %\ignore{\comm{(is this the right reference?)}} 
there exists %a positive constant 
$s_2>0$ such that $$\mu(E_m\cap E_n)\leq s_2\big(\psi^{\delta}(m_0+m)\psi^{\delta}(m_0+n)+a_{n-m}\psi^{\delta}(m_0+n)+ a_{m_0+n} \psi^{\delta}(m_0+m)\big)$$ for $n>m$. 
By taking $b_n=\psi(m_0+n)^\delta$,
% and $s_4=1$. Hence 
the above lemma implies that $\mu\big(R_T^f(\psi)\big)=\mu(\limsup_{n\to\infty} E_n)\ge\frac{s_1^2}{2s_2} >0$.
\end{proof}

We conclude the section with the proof of Theorem \ref{Thm2}, that is, a passage from positive measure to full measure under the assumption that $f$ and $T$ commute.
%\subsection{$T$ and $f$ commute case (Theorem \ref{Thm2})}
This proof is adapted from \cite{hussain}.

\begin{proof}[Proof of Theorem \ref{Thm2}]
Suppose that \equ{comm} holds.
%\begin{equation}
%T\circ f=f\circ T. \label{Tfcommute}
%\end{equation}
	 Consider the set 
\begin{equation*}
R'(\psi):=\left\{x\in X: \liminf_{n\to\infty}\psi(n)^{-1}d\big(T^nx,f(x)\big)<\infty\right\}
\end{equation*}
Take a point $x\in R'(\psi)\cap f^{-1}X_i$. By definition, there exist $c(x)>0$ and $\{n_k\}_k\subset \mathbb{N}$ so that 
\begin{equation*}
	\psi(n_k)^{-1}d\big(T^{n_k}x,f(x)\big)<c(x)\quad  \forall \, k\geq 1
\end{equation*} Let $s(x)$ be a positive real number such that $B\big(f(x),s(x)\big)\subset X_i$. Since \\ $\psi(n)\to 0$, take $N\in \mathbb{N}$ such that for all $k\geq N$, 
\begin{equation*}
	c(x)\psi(n_k)<s(x)
\end{equation*}
Then $d\big(T^{n_k}x,f(x)\big)<c(x)\psi(n_k)<s(x)$, therefore $T^{n_k}x\in X_i$ for all $k\geq N$. Hence for all $k\geq N$, %\comm{(Need to state clearly where $T\circ f=f\circ T$ is used!)}
\begin{align*}
	d\Big(T^{n_k}(Tx),\big(f(Tx)\big)\Big)\underset{\equ{comm}}=d\Big(T^{n_k}(Tx),T\big(f(x)\big)\Big)=d\Big(T(T^{n_k}x),T\big(f(x)\big)\Big)\\\leq  d\big(T^{n_k}x,f(x)\big)<  c(x)\psi(n_k).
\end{align*}
This implies $R'(\psi)\cap (\bigcup_i f^{-1}X_i)\subset T^{-1}R'(\psi)$, thus $\mu\big(R'(\psi)\smallsetminus T^{-1}R'(\psi)\big)=0$. 
But  $R_T^f(\psi)\subset R'(\psi)$, hence $\mu\big(R'(\psi)\big)>0$, and by the ergodicity of $T$, $\mu\big(R'(\psi)\big)=1$.  

 Now we show that $\mu\big(R_T^f(\psi)\big)=1$. Take a sequence of positive numbers $\{\ell(n):n\geq 1\}$ such that 
 \begin{equation*}
 	\sum_{n=1}^\infty  \frac{\psi(n)}{\ell(n)} =\infty \quad \text{and}\quad \lim_{n\to\infty}\ell(n)=\infty.
 \end{equation*} 
Consider $\widetilde{\psi}(n):=\psi(n) /\ell(n)$; then $R'(\widetilde{\psi})$ has full measure, i.e.  for $\mu$-almost every $x\in X$, 
\begin{equation*}
	\liminf_{n\to\infty}\widetilde{\psi}(n)^{-1}d\big(T^nx,f(x)\big)<\infty.
\end{equation*}
By Egorov's theorem, for any $\varepsilon>0$  there exists $M>0$ such that the set 
\begin{equation*}
	R_M:=\left\{x\in X:\frac{\ell(n) }{\psi(n) }d\big(T^nx,f(x)\big)<M \text{ for infinitely many } n\in \mathbb{N}\right\}
\end{equation*} is of measure at least $1-\varepsilon$. Then $R_M\subset R_T^f(\psi)$, by letting $\ell(n)\to \infty$. Since $\varepsilon$ is arbitrary, it %forces that 
implies $\mu\big(R_T^f(\psi)\big)=1$. 
\end{proof}

%\section{Consequences of pseudo-Markov and \\ quasi-bijective properties}\label{mp}
\section{%Full measure of $R_T^f(\psi)$
Proof of Theorem \ref{thm3}}

{We first prove a local version of Lemma \ref{lemma4}; i.e., fix a ball $B$ with sufficiently small radius  and   estimate $\mu(A_n\cap B)$ for $n$ sufficiently large. }

{\begin{lemma}\ignore{[Local version of Lemma \ref{lemma4}]}\label{local4}
	For any $r < r_0/2$ there exists $n_r\in \mathbb{N}$ such that for any open ball $B=B(x,r)$ in $X$ %with $2r<r_0$, there exists so that for
	and  all $n>n_r$,
	\begin{align*}
		%2^\delta
		\frac{\eta_1}{ \eta_2 20^\delta}\mu(B)\left[\frac{\eta_1^2}{\eta_2}\psi(n)^\delta-(2p)^\delta a_n\right]  \leq \mu(A_n\cap B)  \\
		\leq  \frac{\eta_2^23^\delta}{\eta_1^2  } \mu(B)%{\eta_2(3/2)^\delta}
		\left[\eta_25^\delta\psi(n)^\delta+  a_n(2p)^\delta\right]
	\end{align*}
	%where $C_1,C_2,C_3$ are as in Lemma \ref{lemma3}. %{for $\varepsilon=\frac{1}{10}$}. 
\end{lemma}
%\comm{Or simply should I write 
%\begin{equation*}
%	\mu(A_n\cap B)\asymp \mu(B)\psi(n)^\delta
%\end{equation*} since this is all I will use. }
\begin{proof} %\comm{Where is $n_1$? 
%{Let $\varepsilon=\frac{1}{10}$.} 
Let $n_r\in \mathbb{N}$ be such that %$\psi(n)<\frac{rp}{10}$ for all $n>n_r$; i.e., we have 
\eq{ballsinside}{\frac{5\psi(n)}{p}<\frac{r}{2}} for all $n>n_r$.
	%Let $0<\varepsilon<\frac{1}{2}$. 
	As in the proof of Lemma \ref{lemma4}, we have an open covering 
	\begin{equation*}
		\left\{B\Big(y, \frac{\psi(n)}{2p}\Big):y\in B\right\}
	\end{equation*}
	of $B$, with each $B\Big(y, \frac{\psi(n)}{2p}\Big)\subset f^{-1}B\big(x,  \psi(n)/2\big)$ for some $x\in X$. 
	%\comm{Replace $\varepsilon$ by $1/2$. $\varepsilon$ should be $\frac{1}{10}$, which will be taken in the next paragraph.} 
	Again by Vitali's Covering Theorem, %($5r$-covering Lemma), 
	we can find countable sub-collection of disjoint balls $\left\{B\big(y_j, \psi(n)/2p\big)\right\}_{j\in \mathcal{J}}$ such that 
	\begin{equation} 
		B\subset \bigcup_{j\in \mathcal{J}}B\big(y_j, 5\psi(n)/2p\big). \label{cover5r}
	\end{equation} Then we have 
	\begin{equation*}
		\sum_{\substack{j\in \mathcal{J}}} \eta_1\big( \psi(n)/2p\big)^\delta \leq \sum_{\substack{j\in \mathcal{J}}} \mu\Big(B\big(y_j, \psi(n)/2p\big)\Big)\underset{\equ{ballsinside}}\leq \mu\big(B(x,2r)\big)\leq 2^\delta \frac{\eta_2}{\eta_1} \mu(B),
	\end{equation*}
	{hence \begin{equation}\label{upperbound}
	|\mathcal{J}|\leq \frac{\eta_2(4p)^\delta \mu(B)}{\eta_1^2  \psi(n)^\delta}.
	\end{equation} } On the other hand, %by \eqref{cover5r} we have 
	\begin{align*}
		%&
		\sum_{\substack{j\in \mathcal{J}\\ B\big(y, \frac{5\psi(n)}{2p}\big)\cap B(x,r/2)\neq \varnothing}} \eta_2 \big(5\psi(n)/2p\big)^\delta %\\
		\geq \ & \sum_{\substack{j\in \mathcal{J}\\ B\big(y, \frac{5\psi(n)}{2p}\big)\cap B(x,r/2)\neq \varnothing}} \mu\Big(B\big(y_j,5\psi(n)/2p\big)\Big)\\
		\underset{\eqref{cover5r}}\geq &\   \mu\big(B(x,{r}/{2})\big)%\\
		\geq %&
		 \frac{\eta_1}{2^\delta\eta_2} \mu(B),
	\end{align*}
	therefore 
	\begin{equation}\label{lowerbound}
	\begin{aligned}
		&\left|\left\{j\in \mathcal{J}:B\Big(y_j,\frac{\psi(n)}{2p}\Big)\subset B\right\}\right|\\
		\underset{\equ{ballsinside}}\geq \ &\left|\left\{j\in \mathcal{J}:B\Big(y_j,\frac{\psi(n)}{2p}\Big)\cap B(x,r/2)\neq \varnothing \right\}\right|
		\geq  \ \frac{\eta_1p^\delta\mu(B)}{ \eta_2^2 \big(5\psi(n)\big)^\delta}.
	\end{aligned}\end{equation} 
	Since for each $j\in \mathcal{J}$ we have $B \big(y_j,\psi(n)/2p \big)\subset f^{-1}B \big(x_j,\psi(n)/2 \big)$ and %\linebreak 
	 $B \big(y_j,5\psi(n)/2p \big)\subset f^{-1}B \big(x_j,5\psi(n)/2 \big)$,  similarly to the proof  of Lemma \ref{lemma4} we can write
	 \begin{align*}
		\mu(A_n\cap B) \leq & \sum_{j\in \mathcal{J}}\mu\big(B(y_j,5\psi(n)/2p)\cap A_n\big)\\
		\underset{\text{Lemma \ref{lemma3}}}\leq & \ \sum_{j\in \mathcal{J}}\
		{\eta_2(3/2)^\delta} \left[\mu\Big(B\big(y_j,5\psi(n)/2p\big)\Big)+ a_n\right]\psi(n)^\delta\\
		\underset{\eqref{upperbound},\,\eqref{ar}}\leq  \, &\  \frac{\eta_2(4p)^\delta \mu(B)}{\eta_1^2  \psi(n)^\delta} {\eta_2(3/2)^\delta}\left[\eta_2\big(5\psi(n)/2p\big)^\delta\psi(n)^\delta+ a_n\psi(n)^\delta\right] \\
		{=}  \ \  \  \ &\  \frac{\eta_2^23^\delta}{\eta_1^2  } \mu(B)%{\eta_2(3/2)^\delta}
		\left[\eta_25^\delta\psi(n)^\delta+  a_n(2p)^\delta\right]
	\end{align*}
	and 
	\begin{align*}
		\mu(A_n\cap B)\geq & \sum_{\substack{j\in \mathcal{J}\\B(y_j,\psi(n)/2p)\subset B}}\mu\big(B(y_j, \psi(n)/2p)\cap A_n\big) \\
		\underset{\text{Lemma }\ref{lemma3}}\geq & \sum_{\substack{j\in \mathcal{J}\\B(y_j,\psi(n)/2p)\subset B}}
		2^{-\delta}\left[\eta_1\mu\big(B(y_j,\psi(n)/2p)\big)-\eta_2a_n\right] \psi(n)^\delta
		%2^{-\delta} \left[{\eta_1\mu\big(B(y_j,\psi(n)/2p)\big)-\eta_2a_n\right] \psi(n)^\delta
		\\
	\underset{\eqref{lowerbound},\,\eqref{ar}}\geq &\  %\frac{\eta_1\mu(B)}{2^\delta\eta_2^2 (5\psi(n)/2p)^\delta}
	\frac{\eta_1p^\delta\mu(B)}{ \eta_2^2 \big(5\psi(n)\big)^\delta}
	\ 2^{-\delta}\left[\eta_1^2\big(\psi(n)/2p)\big)^\delta-\eta_2a_n\right] \psi(n)^\delta \\
		=\ \ \ \  &\frac{\eta_1}{ \eta_2 20^\delta}\mu(B)\left[\frac{\eta_1^2}{\eta_2}\psi(n)^\delta-(2p)^\delta a_n\right], 	\end{align*}	%{where $C_1,C_2,C_3,C_4$ are defined in Lemma \ref{lemma3} with $\varepsilon=\frac{1}{10}$.} 
finishing the proof. \end{proof}
\begin{lemma}\label{lemma3.7}
	Let $m>m_0$, and let $J$ be a cylinder in $\mathcal{F}_m$. Then \linebreak $\diameter(J)\leq K_1\diameter(X)K_{J}^{-1}$.
% and $\mu(J)\lesssim K_{J}^{-\delta}$.  
\end{lemma}}
%\comm{Where is it used? need to indicate it, or remove.}

\begin{proof}
	Let $x,y\in J$. By \eqref{eq10} and the definition of $K_J$,
	%Assumptions 4 and 5, 
	\begin{equation*}
		 d(x,y)\leq \frac{K_1d(T^mx,T^my)}{K_{J}},
	\end{equation*}
	hence by \eqref{eq10}
	\begin{equation*}
		 d(x,y) \leq K_1\diameter (X)K_{J}^{-1},
	\end{equation*} 
	which implies the needed upper bound on $\diameter (J)$.% \par 
	%Moreover, by the
	% assumption $m>m_0$ we have $K_1\diameter(X)K_{J}^{-1} < r_0$,
	%hence $$ J\subset B\left(x_0', K_1\diameter(X)K_{J}^{-1}\right)$$ for some $x_0,x_0'\in X$. 
	%Then by \eqref{ar}, $\mu(J)\lesssim K_{J}^{-\delta}$.
	\end{proof}

%\smallskip}

{For the rest of the section, let us assume} that $\{X_i\}_{i\in \mathcal{I}}$ is pseudo-Markov. 
%and $T$ is quasi-bijective. 
Let $0<\tau<1$ be such that $\mu(TX_i)\geq \tau\mu(X)$ for all $i\in \mathcal{I}$.

%\comm{This is a new lemma.}

\begin{lemma}\label{equality} %\comm{Need to assume pseudo-Markov.}
	For all $n\in \mathbb{N}$ and for a nonempty cylinder $$J=X_{i_0}\cap T^{-1}X_{i_1}\cap \cdots \cap T^{-n+1}X_{i_n}\in \mathcal{F}_n,$$ 
	one has $T^nJ= TX_{i_n}$. 
\end{lemma}

\begin{proof}
	 For each  $0\leq j<n$, we have
	$T^{-j+1}X_{i_j}\cap T^{-j}X_{i_{j+1}}\neq \varnothing$, and then $TX_{i_j}\cap X_{i_{j+1}}\neq \varnothing$, so by the pseudo-Markov condition, $TX_{i_j}\supset X_{i_{j+1}}$. Now let $x\in X_{i_n}$, then since $TX_{i_{n-1}}\supset X_{i_n}$, there exists some $x_{n-1}\in X_{i_{n-1}}$ so that $Tx_{n-1}=x$. Similarly, there exists some $x_{n-2}\in X_{i_{n-2}}$ such that $Tx_{n-2}=x_{n-1}$. Continue this process until we find such $x_0\in X_{i_0}$. Then $T^{j-1}x_0\in X_{i_j}$ for each $0<j<n$, so in particular $x_0\in J$ and $T^{n-1}x_0=x$. Hence $T^{n-1}J\supset X_{i_n}$.\par For the reverse containment, since $J\subset T^{-n+1}X_{i_n}$, it follows that\linebreak $T^{n-1}J\subset X_{i_n}$. 
	 It remains to apply $T$ to both sides of $T^{n-1}J=X_{i_n}$ to conclude the lemma.
		\end{proof}
		
%\comm{The next two lemmas should be combined into one, maybe without $\mathsf{diam}$.}

\begin{lemma}\label{lemma4.3}
	 Let $m>m_0$, and let $J$ be a cylinder in $\mathcal{F}_m$. 
Then \linebreak $\mu(J)\gtrsim K_{J}^{-\delta}$. 
\end{lemma}
	\begin{proof}
	\ignore{Let $\varepsilon>0$.  Since $T$ is quasi-bijective, there exists some $0<\tau<1$ so that $\mathsf{diam}(TX_i)\geq \tau \mathsf{diam}(X), \mu(TX_i)\geq \tau \mu(X)$ for all $i\in \mathcal{I}$. Since $\{X_i\}_{i\in \mathcal{I}}$ is pseudo-Markov, b{y Lemma \ref{equality}} there exists some $i\in \mathcal{I}$ so that $T^m J=TX_i$. Then there exist $x,y\in X_i$ so that $d(T^mx,T^my)\geq \mathsf{diam}(TX_i)-\varepsilon$.  Since $T$ is quasi-bijective, 
\begin{equation*}
	d(x,y)\geq \frac{K_1^{-1}d(T^mx,T^my)}{K_J},
\end{equation*}
so applying \eqref{eq10} again
\begin{equation*}
	d(x,y)\geq K_1^{-1}(\tau\mathsf{diam}(X)-\varepsilon)K_J^{-1}. 
\end{equation*} Since $\varepsilon$ is chosen arbitrarily, the lower bound on $\mathsf{diam}(J)$ is proved.\par }
	%For the lower bound on $\mu(J)$,  
	Write $J$ in the form
	\begin{equation*}
		J=X_{i_1}\cap T^{-1}X_{i_2}\cap \cdots T^{-m+1}X_{i_m};
	\end{equation*}
then
	\begin{equation*}
		K_J^\delta\mu(J)\underset{\text{Lemma }\ref{lemma6}}\asymp \mu(T^mJ)\underset{\text{Lemma }\ref{equality}}=
	\mu(TX_{i_m})  \geq \tau\mu(X),\end{equation*}
	%where 
		thus $\mu(J)\gtrsim K_J^{-\delta}$. 
\end{proof}

\ignore{{\begin{remark} \rm  \comm{Maybe remove it.}
	By the definition of cylinders, we know for all $m<n$ and for all $I\in \mathcal{F}_n$, 
	\begin{equation*}
		I=\bigsqcup_{\substack{J\in \mathcal{F}_m\\ J\subset I}}J,
	\end{equation*}
	so by Lemma \ref{lemma3.7}, when the partition is pseudo-Markov and $T$ is quasi-bijective
	\begin{equation*}
		K_I^{-\delta}\asymp\sum_{\substack{J\in \mathcal{F}_m\\ J\subset I}}K_J^{-\delta}. 
	\end{equation*}
\end{remark}}}

%\subsection{Full measure: pseudo-Markov and quasi-bijective case}

%\subsubsection{Local estimate of targets and their quasi-independence  \comm{(?)}}
We now prove a local estimate for the quasi-independence of the intersection of sets $\{A_n\}_n$ with balls. 

\begin{corollary}\label{localqi}
	 For any open ball $B=B(x,r)$ in $X$ with $2r<r_0$, there exists $m_r\in \mathbb{N}$ so that for all $n>m> m_r$, 
	\begin{equation*}
		\mu(B\cap A_m\cap A_n)\lesssim \mu(B)\left[\psi(m)^\delta\psi(n)^\delta+a_{n-m}\psi(n)^\delta+a_n\psi(m)^\delta\right],
	\end{equation*} 
	where the implicit constant   in the  above inequality is independent of $B$. 
\end{corollary}

\begin{proof}
	Let $q>m_0$ and let $I\in \mathcal{F}_q$.
	%, where $m_0$ is defined at the beginning of $\S$\ref{mp}. 
	For all $n>m\geq q$, by Lemma \ref{localqilemma} we get 
\begin{align}
	&\mu(I\cap A_m\cap A_n)\notag 
	=  \sum_{\substack{J\in \mathcal{F}_m\\ J\subset I}} \mu(J\cap A_m\cap A_n)\notag 
	\lesssim  \sum_{\substack{J\in \mathcal{F}_m\\ J\subset I}} \mu(J^*\cap A_n)\notag\\
	\lesssim & \sum_{\substack{J\in \mathcal{F}_m\\ J\subset I}}K_{J}^{-\delta} \left[[\psi^\delta(m)\psi(n)^\delta+a_{n-m}\psi(n)^\delta]+ [\psi(m)^\delta\psi(n)^\delta+a_n\psi(m)^\delta ]\right]\notag\\
	\underset{\text{Lemma }\ref{lemma4.3}}\lesssim & K_I^{-\delta}\left[\psi(m)^\delta\psi(n)^\delta+a_{n-m}\psi(n)^\delta+a_n\psi(m)^\delta\right].\label{localqiest}
\end{align}\par 
Now let $B(x,r)$ be an open ball in $X$ with $r<r_0/2$. By the expanding property \eqref{toinfty}, we can take some $m_r>m_0$ so that for all $q>m_r$, 
\begin{equation*}
	\inf_{J\in \mathcal{F}_q}K_J>r^{-1}K_1\diameter(X);
\end{equation*}
Let $q>m_r$. Then by Lemma \ref{lemma3.7} for all $J\in \mathcal{F}_q$, $\diameter(J)<r$, so 
\begin{equation*}
	B(x,r)\subset  \bigsqcup_{\substack{J\in \mathcal{F}_q,\\ J\cap B(x,r)\neq \varnothing}}J \subset B(x,2r). 
\end{equation*}
Then 
\begin{equation*}
 \mu\big(B(x,r)\big)\leq  \sum_{\substack{J\in \mathcal{F}_q,\\ J\cap B(x,r)\neq \varnothing}}\mu(J)\leq  \mu\big(B(x,2r)\big). 
\end{equation*}
Therefore by \eqref{localqiest}, 
\begin{align*}
	\mu\big(B(x,r)\cap A_n\cap A_m\big)\leq & \sum_{\substack{J\in \mathcal{F}_q,\\ J\cap B(x,r)\neq \varnothing}}\mu(J\cap A_n\cap A_m)\\
	\underset{\eqref{localqiest}}\lesssim & \sum_{\substack{J\in \mathcal{F}_q,\\ J\cap B(x,r)\neq \varnothing}} K_I^{-\delta}\left[\psi(m)^\delta\psi(n)^\delta+a_{n-m}\psi(n)^\delta+a_n\psi(m)^\delta\right] \\
	\underset{\text{Lemma }\ref{lemma4.3}}\lesssim & \sum_{\substack{J\in \mathcal{F}_q,\\ J\cap B(x,r)\neq \varnothing}} \mu(J)\left[\psi(m)^\delta\psi(n)^\delta+a_{n-m}\psi(n)^\delta+a_n\psi(m)^\delta\right] \\
	\leq &\  \mu\big(B(x,2r)\big)\left[\psi(m)^\delta\psi(n)^\delta+a_{n-m}\psi(n)^\delta+a_n\psi(m)^\delta\right]\\
	 \leq &\  2^\delta \frac{\eta_2}{\eta_1}\mu\big(B(x,r)\big)\left[\psi(m)^\delta\psi(n)^\delta+a_{n-m}\psi(n)^\delta+a_n\psi(m)^\delta\right].
\end{align*}
\end{proof}

%\subsubsection{{full measure of $R_T^f(\psi)$}}\label{pf}

%\comm{This section is new.}

Finally let us apply the following generalization of the {Lebesgue Density Theorem} to %prove the full measure of $R_T^f(\psi)$ when $\sum_{n=1}^\infty \psi(n)=\infty$. 
finish the proof of Theorem \ref{thm3}. Recall that a  probability measure $\mu$ on $X$ is called \textsl{doubling} if there exists a constant $D>0$ so that for any $x\in X$ and $r>0$, 
	\begin{equation*}
		\mu(B(x,2r))\leq D\mu(B(x,r)).
	\end{equation*}
%\end{definition}
It is clear that  Ahlfors regular %\eqref{ar} of $\mu$, $\mu$ is 
measures are doubling.

\begin{theorem}[Lebesgue Density Theorem] \label{ldt}
	Let $(X,d)$ be a metric space with a Borel doubling probability measure $\mu$, and let $E$ be a Borel subset of $X$. Suppose there exist constants $C>0$ and $r_0>0$ so that for all balls $B\subset X$ with radii less than $r_0$, we have 
	\begin{equation*}
		\mu(E\cap B)\geq C\mu(B).
	\end{equation*} Then $\mu(E)=1$. 
\end{theorem} 

For a proof, see \cite[\S8, Lemma 7]{BDV}. 

\smallskip

\begin{proof}[Proof of Theorem \ref{thm3}]
 %Now we apply Lemma \ref{posmeas}.	
 By Lemma \ref{local4} and Corollary \ref{localqi}, there exist positive constants $s_1,s_2,s_3$ so that for all $B=B(x,r)$ in $X$ with $r<r_0/2$ and for all $n>\max\{n_r,m_r\}$,
	\begin{equation*}
		\mu(B)s_1\left(\psi(n)^\delta-a_n\right)\leq \mu(B\cap A_n)\leq \mu(B)s_3\left(\psi(n)^\delta+a_n\right).
	\end{equation*}
	\begin{equation*}
		\mu(B\cap A_n\cap A_m)\leq \mu(B)s_2\left(\psi(m)^\delta\psi(n)^\delta+a_{n-m}\psi(n)^\delta+a_n\psi(m)^\delta\right).
	\end{equation*}
	Now take $b_n=\psi(n)^\delta$ and
	\begin{equation*}
		E_k=\begin{cases}
			A_k\cap B & \text{if }k>\max\{n_r,m_r\}\\
			\varnothing & \text{otherwise.}
		\end{cases}
	\end{equation*}
	%  and $s_4=\mu(B)$. Then by Lemma \ref{posmeas}, 
	 \begin{equation*}
	 \mu(\limsup_{n\to\infty}A_n\cap B)	=\mu(\limsup_{k\to\infty}E_k)\geq \frac{\big(\mu(B)s_1)^2}{2s_2\mu(B)} =  \frac{s_1^2}{2s_2}\mu(B). 
	 \end{equation*}
	 Then by the Lebesgue Density Theorem (Theorem \ref{ldt}), 
	 \begin{equation*}
	\mu\big(R_T^f(\psi)\big)=\mu(\limsup_{n\to\infty}A_n)=1.
	\end{equation*}
\end{proof}

\ignore{To apply Theorem \ref{ldt}, we first prove the following consequence of the Paley-Zygmund inequality. 

\begin{proof}
	The proof of this lemma uses ideas introduced in \cite[\S2.3.1]{hussain}. Define 
	\begin{equation*}
		f_N=\sum_{n=1}^N \chi_{E_n}.
	\end{equation*} Let $0\leq \theta\leq 1$. In particular, since $\sum_{n=1}^\infty \mu(E_n)=\infty$, if 
	\begin{equation*}
		f_N(x)\geq \theta \sum_{n=1}^N \mu(E_n)\quad \text{infinitely often,}
	\end{equation*}
	then $x\in E_n$ infinitely often; i.e., 
	\begin{equation*}
		\limsup_{N\to\infty}\left\{x:f_N(x)\geq \theta\sum_{n=1}^N \mu(E_n)\right\}\subset \limsup_{n\to\infty}E_n
	\end{equation*}
	Then by the Paley-Zygmund inequality, which states 
	\begin{equation*}
		\mathbb{P}(Z>\theta \mathbb{E}[Z])\geq (1-\theta)^2 \frac{\mathbb{E}[Z]^2}{\mathbb{E}[Z^2]}
	\end{equation*} for all random variables $Z$ with finite variance, 
	we have 
	\begin{align*}
		\mu\left(\limsup_{n\to\infty}E_n\right)\geq & \mu\left(\limsup_{N\to\infty}\left\{x:f_N(x)\geq \theta\sum_{n=1}^N \mu(E_n)\right\}\right) \\
		\geq & (1-\theta)^2 \limsup_{N\to\infty} \frac{\left(\int f_n\,d\mu\right)^2}{\int f_n^2\,d\mu}\\
		=& (1-\theta)^2 \limsup_{N\to \infty} \frac{\left(\sum_{n=1}^N \mu(E_n)\right)^2}{\sum_{n,m=1}^N \mu(E_n\cap E_m)}.
	\end{align*}
	By taking $\theta\to 0$, we prove the lemma. 
\end{proof}}

\section{Examples}\label{example}

Here we list several examples of dynamical systems to which our theorems apply. The first two come from the paper \cite{hussain}:
 \begin{itemize}
		\item  $X=[0,1]$, $T:x\mapsto \beta x\mod{1}$, where $\beta > 1$, and $\mu$ is the $T$-invariant probability measure absolutely continuous with respect to Lebesgue measure, namely (see \cite{renyi})
		\begin{equation*}
	\mu(E)=\begin{cases}
	\Leb(E) & \text{if }\beta \text{ is an integer,}\\
	\frac{1}{\sum_{k=0}^\infty \frac{\{\beta^k\}}{\beta^k}}\sum_{k=0}^\infty\frac{\Leb\left(E\cap [0,\{\beta^k\}]\right)}{\beta^k} & \text{if }\beta \text{ is not an integer,}
	\end{cases}
\end{equation*} 
 where   $\{x\}$ denotes the fractional part of $x$;
\medskip
		\item  $X=[0,1]$, $T:x\mapsto\frac{1}{x}\mod{1}$, and $\mu$ is the Gauss measure given by ${d\mu} =\frac{dx}{(\log 2)(1+x)}$.
	\end{itemize}
Sections 3.1--3.2 of \cite{hussain} together with Remark \ref{comparison} show  that in  both cases   the assumptions of Theorem \ref{thm2.4} are satisfied.  In fact,  in  both cases uniform mixing   with exponential rate was first exhibited in  \cite{Ph}, together with a quantitative shrinking target property of these systems.

%\comm{This section needs to be rewritten considerably. Right now it is repetitive and basically recycles known results. 
%Here is a suggested rough draft. Maybe you can finish it off.}

{The pseudo-Markov property holds for the Gauss map but only for some special $\beta$-transformations. We will prove the full measure in the divergence case for arbitrary $\beta$-transformations in \S\ref{beta}.}

\smallskip
	
%\subsection{Linear iterated functions systems}

Our last example deals with self-similar sets. Let %$d$ be the Euclidean distance on $\R$ and let 
$$\Theta:=\{\theta_i(x):\mathbb{R}\to \mathbb{R}\}_{i=1}^L$$ be a system of similarities with 
\begin{equation*}
	|\theta_i(x)-\theta_i(y)|=r_i |x-y| \text{ for all } x,y\in \mathbb{R},\  
	 i=1,\ldots, L,
\end{equation*}
where $0<r_i<1$ for all $i$.
%We denote $$r_{\min}=\min_{i=1,\ldots,L}r_i\text{ and }r_{\max}=\max_{i=1,\ldots,L}r_i.$$ 
Then by {\cite[Theorem 3.1.(3)]{hutchinson}} there exists a unique nonempty compact set $X\subset \R$, called the {\sl attractor} of the system, such that $$X=\bigcup_{i=1}^L \theta_i(X).$$ Furthermore, we assume that $\Theta$ satisfies the {\sl open set condition}: that is, there exists a non-empty bounded open set $U\subset \mathbb{R}$ such that 
 \begin{equation*}\label{osc}
	\bigcup_{i=1}^L \theta_i(U) \subset U\text{ and } \theta_i(U)\cap \theta_j(U)=\varnothing \text{ for all }i\neq j.
\end{equation*} Then it is known that the Hausdorff dimension of $X$ is equal to the unique solution $\delta\in[0,1]$ of the equation  $\sum_{i=1}^L r_i^\delta=1$ {(see \cite[Theorem 9.3]{falconer})}. Furthermore, the normalized restriction  $\mu$ of the $\delta$-dimensional Hausdorff measure to $X$ is positive, finite and satisfies %$\Theta$-invariant; i.e., 
\eq{mu}
{\mu=%\frac{1}{L}
\sum_{i=1}^L r^\delta_i \cdot(\theta_i)_*\mu.}
%{where $(\theta_i)_*\mu(E)=\mu(\theta_i^{-1}E)$ for all measurable $E$ \mynew{(Or maybe this is a very obvious notation then please ignore this.)}
(For a proof, see \cite[Theorem 4.4.(1)]{hutchinson}.) \ignore{\comm{Please add exact references for both claims, perhaps both can refer to the original paper \cite{hutchinson}.}}

\smallskip

To {define the corresponding expanding map and} construct the cylinders, we consider the following lemma from \cite{Sc,Gr}:
	
	{\begin{lemma}[\cite{Sc}, Theorem 2.2; \cite{Gr}, Lemma 3.3] \label{lemma5.1}
		Let $%\Theta:=
		\{\theta_i:\mathbb{R}^n\to \mathbb{R}^n\}_{i=1}^L$ be a system of similarities  satisfying the open set condition, $X$  its attractor, and $\mu$ the self-similar measure given by \equ{mu}. Then there exists a nonempty compact set $A$ with 
		\begin{itemize}
%			\item[(i)] $X\subset A$;
%	\smallskip
%			\item[(ii)] $\overline{\mathsf{int}(A)}=A$;
%	
			\item[(i)] $\theta_i(A)\subset A$ for all $i=1,\ldots,L$;
				\smallskip
			\item[(ii)] $\theta_i\big(\mathsf{int}(A)\big)\cap \theta_j\big(\mathsf{int}(A)\big)=\varnothing$ for each $i\neq j$;
			\smallskip
			\item[(iii)] $\mu\big(\mathsf{int}(A)\cap X
			\big)=1$.
		\end{itemize}
	\end{lemma}}
	
	We remark  that parts (i) and (ii) are stated in  \cite[Theorem 2.2]{Sc}, and part (iii) follows from the proof of  \cite[Lemma 3.3]{Gr}, where it is shown that $\mu\big((A\smallsetminus \mathsf{int}(A)%\cap X
	\big)=0$ and $X = \supp\mu\subset A$.
\smallskip

	Now define 
		\begin{equation}\label{order1}
			X_i:=\theta_i\big(\mathsf{int} (A)\cap X\big).
		\end{equation} Each $X_i$ is open in $X$ because $\theta_i$ is an open map.  The disjointness of $X_i$ and $X_j$ for $i\neq j$ follows from Lemma 	\ref{lemma5.1}(ii). Finally, one can write		\begin{equation*}
		\begin{aligned}
			\mu\left(\bigcup_{i=1}^L X_i\right)&=\sum_{i=1}^L \mu\big(\theta_i(\mathsf{int} (A)\cap X)\big)=\sum_{i=1}^L r_i^\delta \mu(\mathsf{int} (A)\cap X)\\ \text{(by Lemma \ref{lemma5.1}(iii))\ }&=\sum_{i=1}^L r_i^\delta= 1 = \mu(X).
			\end{aligned}
		\end{equation*} %whereas the second equality holds because each $\theta_i$ is bi-Lipshitz. 
		{Hence one can define the map $T:X\to X$ $\mu$-almost everywhere by setting 
		%In addition, 
		$T|_{X_i}=\theta_i^{-1}|_{X_i}$. {It follows from \equ{mu} that} $(X, \mu,T)$ is a   measure-preserving system. Clearly} $T|_{X_i}$ is continuous and injective {for every $i$.
		Therefore}
		the collection $\{X_i\}_{i=1}^L$ satisfies our assumption for being cylinders of order $1$. 
	\smallskip
	
		For  $\mathbf{i} = (i_1,\dots,i_m)\in \{1,\ldots,L\}^m$  let us define $$\theta_{\mathbf{i}} := \theta_{i_1}\circ   \cdots \circ \theta_{i_m}{\quad\text{and}\quad r_{\mathbf i} := \prod_{k=1}^m r_{i_k}}.$$
		Using the definition \eqref{cylinders} of cylinders of order $m$ together with   % \eqref{pi} and
		 \eqref{order1}, it is easy to see that the set $\mathcal F_m$ of cylinders of order $m$ is precisely
		$$
\big\{X_{\mathbf i} := \theta_{\mathbf i}\big(\mathsf{int} (A)\cap X\big) : \mathbf{i}  \in \{1,\ldots,L\}^m\big\},
$$
and the restriction of $T^m$ onto $X_{\mathbf i}\in \mathcal F_m$ coincides with $\theta_{\mathbf i}^{-1}$. This, %with the notation
%$$r_{\mathbf i} := \prod_{k=1}^m r_{i_k}\quad\text{for}\quad\mathbf{i} = (i_1,\dots,i_m)\in \{1,\ldots,L\}^m,$$ 
in particular, implies that
\eq{measure}{\mu(X_{\mathbf i} ) = r_{\mathbf i}^\delta} and \eq{uniform}{K_J = \frac{|T^mx-T^my|}{|x-y|} = r_{\mathbf i}^{-1}\quad\text{for all }  x,y\in J = X_{\mathbf i}\in \mathcal F_m.}	
\smallskip

	\smallskip
%Now we verify that this system satisfies the assumptions of Theorem \ref{thm2.4}. 
	
	%	\mycomm{\begin{proof}
%		(i) follows from the proof of Lemma 3.3, \cite{Gr} that $X=\mathsf{supp}\mu\subset J$. (ii), (iv) and (v) are from Theorem 2.2, \cite{Sc}. (iii) follows from the proof of Lemma 3.3, \cite{Gr} that $\mu((A\backslash \mathsf{int}A)\cap X)=0$. Note that (iii) is trivial in the $\mathbb{R}^1$ case because $A\backslash \mathsf{int} A$ is countable so it has Hausdorff measure zero.  
%	\end{proof}}

Denote $$r_{\max}:= \max_{i = 1,\dots,L}r_i%\quad\text{and}\quad r_{\min}:= \min_{i = 1,\dots,L}r_i
.$$ Let us now verify assumptions \eqref{ar}--\eqref{conf} of Theorem \ref{thm2.4}.

\begin{enumerate}
	\item[\eqref{ar}:] By   \cite[Theorem 5.3(1)(i)]{hutchinson}, 
	\begin{equation*}
		\gamma_1\leq \liminf_{r\to 0}\frac{\mu\big(B(x,r)\big)}{r^\delta}\leq \limsup_{r\to 0}\frac{\mu\big(B(x,r)\big)}{r^\delta}\leq \gamma_2
	\end{equation*} for some $0<\gamma_1<\gamma_2<\infty$ and all $x\in X$. Clearly it implies that $\mu$ is Ahlfors regular of dimension $\delta$.
	\ignore{Then there exists $r_0>0$ so that for all $r<r_0$,
	\begin{equation*}
		1/2\gamma_1r^\delta<\mu\big(B(x,r)\big)<2\gamma_2 r^\delta
	\end{equation*}}
	\smallskip

	\item[\eqref{em}:] ($T$ is uniformly mixing) Let $E$ be a non-empty open ball in $X$, let $F$ be a measurable set in $X$, and let $m\in \mathbb{N}$. Note that for all cylinders $J= X_{\mathbf i}\in \mathcal{F}_m$, where $\mathbf{i}  \in \{1,\ldots,L\}^m$ one can write 
	\begin{equation*}
	\mu(J\cap T^{-m}F)=r_{\mathbf i}^\delta\mu(T^{m}J\cap F)=r_{\mathbf i}^\delta\mu\big(\mathsf{int} (A)\cap F\big) \underset{\equ{measure}}= \mu(J)\mu(F).
	\end{equation*} 
	Note that since $E$ is an interval, we can (up to a set of measure zero) write $E$ as a disjoint union of cylinders of order $m$ and at most two balls contained in cylinders of order $m$; i.e., $$E=\Big(\bigcup_{ J\in \mathcal{F}_{m},\,J\subset E}J\Big) \cup E_1 \cup E_2$$ where the unions above are disjoint, and $E_1,E_2$ are contained in some cylinders $J_1,J_2\in \mathcal{F}_{m}$ respectively, hence have measure  not greater than %$\frac{1}{L^n}$
	$r_{\max}^{m\delta}$. Then 
	\begin{align*}
		&\left|\mu(E\cap T^{-{m}}F)-\mu(E)\mu(F)\right|\\
		=\ &\Big|\sum_{\substack{J\in \mathcal{F}_{m},\\ J\subset E}} (J\cap T^{-{m}}F)-\sum_{\substack{J\in \mathcal{F}_{m},\\ J\subset E}} \mu(J)\mu(F)+\mu(E_1\cap T^{-{m}}F)-\mu(E_1)\mu(F)\\
		+\ &\mu(E_2\cap T^{-{m}}F)-\mu(E_2)\mu(F)\Big|\\
		=\ & \Big|\mu(E_1\cap T^{-{m}}F)-\mu(E_1)\mu(F)+\mu(E_2\cap T^{-{m}}F)-\mu(E_2)\mu(F)\Big|\\
		\leq\ & \mu(J_1\cap T^{-{m}}F)+\mu(J_1)\mu(F)+\mu(J_2\cap T^{-{m}}F)+\mu(J_2)\mu(F)\\
		=\ &{4}r_{\max}^{m\delta}\mu(F),
	\end{align*}
	and \eqref{em} follows with $a_n = {4}r_{\max}^{n\delta}$.
	\smallskip

		\item[\eqref{eq10}:] 
		%For all cylinders $J\in \mathcal{F}_m$, $J$ has a finite symbolic coding $j$ of length $m$, and for all $x,y\in J$, $$ d(T^m x,T^m y)=r_{j(1)}^{-1}\cdots r_{j(m)}^{-1}d(x,y).$$ 
		Follows from \equ{uniform} with $K_1 = 1$.
	 \smallskip

\item[\eqref{toinfty}:] Again from \equ{uniform}, for all $m\in \mathbb{N}$ we have
	 \begin{equation*}
	 	\inf_{J\in \mathcal{F}_m}K_J = \inf_{\mathbf{i} \in \{1,\ldots,L\}^m} r_{\mathbf{i}}^{-1} =  r_{\max}^{-m},
	 \end{equation*} which goes to $\infty$ as $m\to \infty$. 
	\smallskip

	 \item[\eqref{qi}:] In view of \equ{uniform}, for all $m\in\N$ one can write
	 \begin{equation*}
	 	%\sup_{m\in \mathbb{N}}
		\sum_{J\in \mathcal{F}_m}K_J^{-\delta} =\sum_{\mathbf{i} \in \{1,\ldots,L\}^m} r_{\mathbf{i}}^{\delta}=%\sup_{m\in \N} 
		(r_1^\delta+\cdots+r_m^\delta)^m=1.
	 \end{equation*}
	\item[\eqref{conf}:] Also follows from \equ{uniform} with $K_2 = 1$.	\end{enumerate} 

	\smallskip

It is clear that $\{X_i\}_{i\in \mathcal{I}}$ is pseudo-Markov. %and $T$ is quasi-bijective. 
Thus, by Theorem \ref{thm3}, for any function $\psi:\mathbb{N}\to \mathbb{R}^+$ and any Lipshitz function $f:X\to X$, the $f$-twisted recurrence set $R_T^f(\psi)$ satisfies \begin{itemize}
		\item  $\mu\big(R_T^f(\psi)\big)=0$ if $\sum_{n=1}^\infty \psi(n)^\delta<\infty$; 
		\item  $\mu\big(R_T^f(\psi)\big)=1$ if $\sum_{n=1}^\infty \psi(n)^\delta=\infty$.
	\end{itemize}

\section{{Proof of Theorem \ref{thmbeta}}}\label{beta}

{Let $\beta>1$ be a real number and suppose $\sum_{n=1}^\infty \psi(n)=\infty$. In this section, we will consider the system 
\begin{equation*}
	X=[0,1],\quad T=M_\beta, \quad \mu \text{ the $M_\beta$-invariant measure,}
\end{equation*}
and the partition 
 \begin{equation}\label{betapar}
 	\left\{X_i=\left(\frac{i}{\beta},\frac{i+1}{\beta}\right):i=0,1,\ldots,\lfloor\beta\rfloor-1\right\}\cup \left\{X_{\lfloor\beta\rfloor}=\left(\frac{\lfloor\beta\rfloor}{\beta},1\right)\right\}.
 \end{equation}}
{If $\beta$ is an integer, then \eqref{betapar} is pseudo-Markov. Furthermore, if $\beta$ satisfies
\begin{equation*}
	\beta^2-k\beta-\ell=0,\quad \text{for some } k\geq \ell \in \mathbb{N}^*
\end{equation*} then 
\begin{equation*}
	TX_i=X\ \ \forall i<\lfloor\beta\rfloor\quad \text{and}\quad TX_{\lfloor\beta\rfloor}=\left(0,\frac{\ell}{\beta}\right),\end{equation*} and hence \eqref{betapar} is also pseudo-Markov. In both cases, we have \linebreak $\mu\left(R_T^f(\psi)\right)=1$. by Theorem \ref{thm3}.} \ignore{Hence we suppose that $\beta$ is not an integer.} We now prove the general case, which is also proved in \cite{LVW}.\par 
{We will apply Lemmas \ref{local4} and \ref{localqilemma}. We remark that analogous results are also proved in \cite{LVW}; however our lemmas are proved in a more general   abstract setting, do not depend on the actual arithmetic and symbolic coding of the systems, and  have a much weaker assumption on the regularity of the cylinders.} 

 \smallskip
 We first state some facts about $(X,\mu,T)$. \begin{lemma}[\cite{renyi}]\label{lemmar}
	For $(X,\mu,T)$ defined above,  $\mu$ is equivalent to $\Leb$ and $\mu(E)=\frac{1}{\sum_{k=0}^\infty \frac{\{\beta^k\}}{\beta^k}}\sum_{k=0}^\infty\frac{\Leb\left(E\cap [0,\{\beta^k\}]\right)}{\beta^k}$ for all Lebesgue-measurable set $E$. 
\end{lemma}

For a proof, see \cite[Theorem 1]{renyi}. \par 
\smallskip
To state the next lemma, we define the \textsl{lexicographical order} $\prec$. For two words $\alpha=(\alpha_1,\alpha_2,\ldots)$ and $\beta=(\beta_1,\beta_2,\ldots)$, we write $\alpha \prec \beta$ if there exists $k\in \mathbb{N}$ so that $\alpha_i=\beta_i$ for all $i\leq k$ and $\alpha_k<\beta_k$. We write $\alpha \preceq \beta$ if $\alpha \prec \beta$ or $\alpha=\beta$. 
Denote the $\beta$-expansion of $1$ by 
	\begin{equation*}
		1=\sum_{n=1}^\infty \frac{\xi_n}{\beta^n}.
	\end{equation*}
If $\{\xi_n\}_n$ is not eventually zero, then we define $\xi_n^*=\xi_n$ for all $n\in \mathbb{N}$. If $\{\xi_n\}_n$ is eventually zero, then 
denote $M_\xi:=\max\{n\in \mathbb{N}:\xi_n\neq 0\}$ and define 
\begin{equation*}
	\xi_n^*=\begin{cases}
		\xi_k & \text{if }n\equiv k\mod{M_\xi},\, k>0\\
		\xi_{M_\xi}-1 & \text{if }n\equiv 0\mod{M_\xi}.
	\end{cases}
\end{equation*}

A classical result says that the right-most cylinder has the maximal coding in lexicographical order, in the following sense:

\begin{lemma}[\cite{parry}]\label{lemmap}
	Let $(X,\mu,T)$ and $\{X_i\}_{i\in \mathcal{I}}$ be defined above, and let 
	\begin{equation*}
		J=X_{i_0}\cap T^{-1}X_{i_1}\cap \cdots \cap T^{-n+1}X_{i_{n-1}}.
	\end{equation*} 
 Then $J$ is nonempty if and only if 
	\begin{equation*}
		(i_j,i_{j+1},\ldots,i_{n-1})\prec (\xi_1^*,\xi_2^*,\ldots), \quad \text{for each }j=0,\ldots, n-1.
	\end{equation*}
	\end{lemma}
	
For a proof, see \cite[Theorem 3]{parry}. \par 
\smallskip

\smallskip
For a cylinder $J\in \mathcal{F}_n$ we always have $\Leb(J)\leq \beta^{-n}$. Let us call $J$ \textsl{full} if the upper bound is reached; i.e., $\Leb(J)=\beta^{-n}$. 

\begin{lemma}[\cite{FW}]\label{lemmafw}
	Let $J$ be as in Lemma \ref{lemmap}. If $J$ is nonempty, then 
	\begin{equation*}
		X_{i_0}\cap T^{-1}X_{i_1}\cap \cdots \cap T^{-n+1}X_{j_{n-1}}
	\end{equation*}
	is full for all $j_{n-1}<i_{n-1}$. 
	\end{lemma} 

For a proof, see \cite[Lemma 3.2(1)]{FW}. 
	
\begin{proposition}\label{fullprop}
	For all $n\in \mathbb{N}$ and for all $J\in \mathcal{F}_n$ with $J\neq \varnothing$, there exists some full cylinder $I\in \mathcal{F}_m$ with 
	\begin{equation*}
		I\subset J \quad \text{and}\quad \Leb(I)\geq \frac{\Leb(J)}{\beta}.
	\end{equation*}
\end{proposition}

{A similar inequality is proved in \cite{LVW}. }

\begin{proof}
 Let $J\in \mathcal{F}_n$ with $J\neq \varnothing$.	By Lemma \ref{lemmap}, the set 
	\begin{equation*}
		\mathcal{N}_J:=\left\{k\in \mathbb{N}:J\cap \bigcap_{j=0}^{k} T^{-n-j+1}X_0\cap T^{-n-k}X_1\neq \varnothing\right\}
	\end{equation*} is nonempty. Let $k_J=\inf \mathcal{N}_J$. If $k_J=0$, then trivially we have \linebreak $
		J\cap \bigcap_{j=0}^{k_J} T^{-n-j+1}X_0=J$. Now suppose $k_J>0$.
	 Assume
	 \begin{equation*}
	J\supsetneq J\cap \bigcap_{j=0}^{k_J} T^{-n-j+1}X_0.
	 \end{equation*} Then there exists some $1\leq i<n-1$ such that 
	 \begin{equation}
	J\cap \bigcap_{j=0}^{k_J-1} T^{-n-j+1}X_0\cap T^{-n-k_J+1}X_i\neq \varnothing. \label{eqbad}
	 \end{equation} 
	 But by Lemma \ref{lemmafw}, $i=1$ must satisfy \eqref{eqbad}, which contradicts the minimality of $k_J$ in $\mathcal{N}_J$, so 
	 \begin{equation*}
	 	J\cap \bigcap_{j=0}^{k_J} T^{-n-j+1}X_0=J. 
	 \end{equation*} 
	 By Lemma \ref{lemmafw}
	\begin{equation*}
		\Leb\left( J\cap \bigcap_{j=0}^{k_J+1} T^{-n-j+1}X_0\right)=\beta^{-n-k-1}.
	\end{equation*}
	On the other hand, 
	\begin{equation*}
		 \Leb\left(J\cap \bigcap_{j=0}^{k_J} T^{-n-j+1}X_0\right)\leq \beta^{-n-k}, 
	\end{equation*} hence
	\begin{equation*}
		\Leb(J)\leq \beta \Leb\left( J\cap \bigcap_{j=0}^{k_J+1} T^{-n-j+1}X_0\right). 
	\end{equation*}
	Taking $I=J\cap \bigcap_{j=0}^{k_J+1} T^{-n-j+1}X_0$, we finish the proof. 
\end{proof}

\ignore{We then denote the full cylinder $ J\cap \bigcap_{j=0}^{k_J+1} T^{-n-j+1}X_0$ by $I_J$. }

\smallskip

Note that $K_J=\beta^n$ for all $J\in \mathcal{F}_n$. Now we can prove Theorem \ref{thmbeta}. 

\begin{proof}[Proof of Theorem \ref{thmbeta}]
	Sections 3.1--3.2 of \cite{hussain} together with Remark \ref{comparison} show that the system satisfies \eqref{ar}--\eqref{conf}. By Lemma \ref{local4}, there exists positive constants $s_1,s_3$ so that for all $B=(x,r)$ with $r<\frac{1}{2}$ %in $X$ 
	and for all $n>n_0$, 
	\begin{equation}\label{eq7.1}
		s_1\mu(B)\left(\psi(n)^\delta-a_n\right)\leq \mu(B\cap A_n)\leq s_3\mu(B)\left(\psi(n)^\delta+a_n\right).
	\end{equation} \par 
	By Lemma \ref{lemmar}, there exist a constant $\alpha>0$ such that 
	\begin{equation*}
		\frac{1}{\alpha}\Leb(E)\leq \mu(E)\leq \alpha \Leb(E), \quad \text{for all  Lebesgue-measurable  }E\subset X.
	\end{equation*} Suppose $I\in \mathcal{F}_q$ and $\Leb(I)=\beta^{-q}$. Then by Lemma \ref{localqilemma}, for all $n>m>q>m_0$, 
	\begin{align*}
		\mu(I\cap A_m\cap A_n)\lesssim &\sum_{\substack{J\in \mathcal{F}_m\\
		J\subset I}}\beta^{-m}\left(\psi(m)^\delta\psi(n)^\delta+a_{n-m}\psi(n)^\delta+a_n\psi(m)^\delta\right)\\
		\leq& \ \Leb(I) \left(\psi(m)^\delta\psi(n)^\delta+a_{n-m}\psi(n)^\delta+a_n\psi(m)^\delta\right)\\
		\leq & \ \alpha \mu(I) \left(\psi(m)^\delta\psi(n)^\delta+a_{n-m}\psi(n)^\delta+a_n\psi(m)^\delta\right),
	\end{align*}
	so there exists a constant $s_2$ such that 
	\begin{equation}\label{eq7.2}
		\mu(I\cap A_m\cap A_n)\leq s_2 \mu(I) \left(\psi(m)^\delta\psi(n)^\delta+a_{n-m}\psi(n)^\delta+a_n\psi(m)^\delta\right).
	\end{equation}\par 
	Note that in this system  all cylinders are intervals, so by \eqref{eq7.1} and \eqref{eq7.2} together with Lemma \ref{posmeas}, there exists a constant $\gamma>0$ so that for all $q>m_0$ and for all cylinders $I\in \mathcal{F}_q$ with $\Leb(I)=\beta^{-q}$, we have 
	\begin{equation*}
		\mu\left(I\cap R_f^T(\psi)\right)\geq \gamma\mu(I). 
	\end{equation*}\par 
By Proposition \ref{fullprop}, for all cylinders $J\in \mathcal{F}_n$, there exists a cylinder $I\subset J$ with $I\in \mathcal{F}_q$, $\Leb(I)=\beta^{-q}$, and $\Leb(I)\geq \frac{1}{\beta}\Leb(J)$. If $n>m_0$, then
\begin{equation*}
	\mu\left(J\cap R_f^T(\psi)\right)\geq \mu\left(I\cap R_f^T(\psi)\right)\geq \gamma\mu(I)\geq \frac{\gamma}{\alpha^2\beta}\mu(J). 
\end{equation*} \par 
Now take $B(x,r)\subset X$ with $r<1$. Take $n\in \mathbb{N}$ with $n>-\log_\beta \frac{r}{4}$. Then for all $J\in \mathcal{F}_n$, $\diameter(J)\leq 2\beta^{-n}< {r}/{2}$, and
\begin{equation*}
	B\left(x, {r}/{2}\right)\subset \bigsqcup_{\substack{J\in \mathcal{F}_n\\ J\cap B\left(x, {r}/{2}\right)\neq \varnothing}}J \subset B(x,r).
\end{equation*}
Hence 
\begin{align*}
	\mu\left(B(x,r)\cap R_f^T(\psi)\right)\geq & \ \mu\left( \bigsqcup_{\substack{J\in \mathcal{F}_n\\ J\cap B\left(x, {r}/{2}\right)\neq \varnothing}}J\cap R_f^T(\psi)\right)%\\
	\geq \sum_{\substack{J\in \mathcal{F}_n\\ J\cap B\left(x, {r}/{2}\right)\neq \varnothing}} \frac{\gamma}{\alpha^2\beta}\mu(J)\\
	\geq &  \ \frac{\gamma }{\alpha^2\beta}\mu\big(B(x, {r}/{2})\big)%\\
	\geq  \frac{\gamma %\eta_2
	}{2\alpha^4\beta%\eta_1
	} \mu\big(B(x,r)\big).
\end{align*}
%where $\eta_1,\eta_2$ are given by the Ahlfors regularity \eqref{ar} of $\mu$. 
Thus it follows from the Lebesgue Density Theorem (Theorem \ref{ldt}) that $\mu\big(R_T^f(\psi)\big)=1$.		 \end{proof}

\end{document}